\documentclass[12pt]{amsart}
\usepackage{amsthm,xypic,mathrsfs}
\usepackage{amsmath}
\usepackage[margin=1in]{geometry}
\usepackage{epigraph}
\usepackage{hyperref}
\usepackage{float}
\usepackage{varioref}
\usepackage{enumitem}
\usepackage{pgfplots}
\usepackage{wrapfig,subfig}
\usepackage{float}

\pgfplotsset{compat=1.8}
\tikzset{
	every pin/.style={fill=yellow!50!white,rectangle,rounded corners=3pt},
	small dot/.style={fill=black,circle,inner sep=1pt}
}
\usepgfplotslibrary{fillbetween}
\begin{document}

\newcommand{\A}{{\mathcal A}}
\newcommand{\Alb}{{\text{Alb}}}
\newcommand{\alg}{\text{alg}}
\newcommand{\bk}{{\overline{k}}}
\newcommand{\C}{{\mathbb C}}
\newcommand{\cris}{\text{cris}}
\newcommand{\CRIS}{\text{CRIS}}
\newcommand{\dom}{\operatorname{\text{Dom}}}
\newcommand{\dR}{\text{dR}}
\newcommand{\Ext}{{\rm Ext}\,}
\newcommand{\F}{{\mathbb F}}
\newcommand{\Fil}{\operatorname{\text{Fil}}}
\newcommand{\gr}{\operatorname{\text{gr}}}
\newcommand{\Hom}{\operatorname{\text{Hom}}}
\newcommand{\image}{\operatorname{\text{image}}}
\newcommand{\length}{\operatorname{\text{length}}}
\newcommand{\M}{{\mathcal M}}
\newcommand{\mult}{{\operatorname{\text{mult}}}}
\newcommand{\N}{\operatorname{\mathbb{N}}}
\renewcommand{\O}{{\mathcal O}}
\newcommand{\ord}{{\operatorname{\text{ord}}}}
\renewcommand{\P}{{\mathbb P}}
\newcommand{\PH}{{\mathcal{PH}}}
\newcommand{\Pic}{{\rm Pic\,}}
\newcommand{\PN}{{\mathcal{PN}}}
\newcommand{\PNH}{{\mathcal{PNH}}}
\newcommand{\Q}{{\mathbb Q}}
\newcommand{\R}{{\mathbb R}}
\newcommand{\rank}{{\rm rk}\,}
\newcommand{\sbull}{{\scriptstyle{\bullet}}}
\newcommand{\Spec}{{\text{Spec}}}
\newcommand{\tors}{\text{tors}}
\newcommand{\Z}{{\mathbb Z}}
\newcommand{\bbH}{{\mathbb H}}
\newcommand{\sH}{{\mathscr H}}
\newcommand{\hod}{{\rm Hod}}
\newcommand{\new}{{\rm New}}
\let\liminv=\projlim
\let\isom=\simeq
\let\rk=\rank
\let\tensor=\otimes
\let\into=\hookrightarrow
\let\mydot=\sbull

\newcommand{\be}{\begin{equation}}
\newcommand{\ee}{\end{equation}}
\newcommand{\bes}{\begin{equation*}}
\newcommand{\ees}{\end{equation*}}
\newcommand{\bea}{\begin{eqnarray}}
\newcommand{\eea}{\end{eqnarray}}
\newcommand{\beas}{\begin{eqnarray}}
\newcommand{\eeas}{\end{eqnarray}}
\newcommand{\ben}{\begin{note}}
\newcommand{\een}{\end{note}}
\newcommand{\bexl}{\vskip0.1em\noindent\hrulefill\vskip1em\begin{ExerciseList}}
\newcommand{\eexl}{\end{ExerciseList}\hrulefill}
\newcommand{\exr}{\Exercise}

\newcommand{\bthm}{\begin{theorem}}
\newcommand{\ethm}{\end{theorem}}
\newcommand{\bpro}{\begin{prop}}
\newcommand{\epro}{\end{prop}}
\newcommand{\bcor}{\begin{corollary}}
\newcommand{\ecor}{\end{corollary}}
\newcommand{\bcon}{\begin{conjecture}}
\newcommand{\econ}{\end{conjecture}}
\newcommand{\bp}{\begin{proof}}
\newcommand{\ep}{\end{proof}}
\newcommand{\blem}{\begin{lemma}}
\newcommand{\elem}{\end{lemma}}
\newcommand{\bn}{\begin{note}}
\newcommand{\en}{\end{note}}
\newcommand{\benum}{\begin{enumerate}}
\newcommand{\eenum}{\end{enumerate}}
\newcommand{\bed}{\begin{defn}}
\newcommand{\eed}{\end{defn}}
\newcommand{\brem}{\begin{remark}}
\newcommand{\erem}{\end{remark}}

\newcommand{\btik}{\begin{tikzpicture}\begin{axis}[scale=0.5,axis y line=center, axis x line=middle]}
\newcommand{\etik}{\end{axis}\end{tikzpicture}}

\let\into=\hookrightarrow
\let\mapsto=\longmapsto
\let\cong=\equiv

\newcommand{\upperRomannumeral}[1]{\uppercase\expandafter{\romannumeral#1}}
\newcommand{\citeega}[3]{(see [{\bf EGA}\,\upperRomannumeral{#1}${}_{#2}$, #3])\nocite{ega#1-#2}}

\theoremstyle{definition}
\newtheorem{theorem}[equation]{Theorem}      
\newtheorem{lemma}[equation]{Lemma}          %
\newtheorem{corollary}[equation]{Corollary}  
\newtheorem{proposition}[equation]{Proposition}
\newtheorem{conj}[equation]{Conjecture}
\newtheorem{example}[equation]{Example}
\newtheorem{question}[equation]{Question}
\newtheorem{defn}[equation]{Definition}
\newtheorem{remark}[equation]{Remark}

\numberwithin{equation}{subsection}
\newtheorem*{theorem*}{Theorem}

\title{Crystalline aspects of geography of low dimensional varieties I:
        Numerology}
\author{Kirti Joshi}
\address{Department of Mathematics, University of Arizona, Tucson, AZ, 85721}

\begin{abstract}
This is a modest attempt to study, in a systematic manner, the
structure of low dimensional varieties in positive characteristics using $p$-adic invariants.
The main objects of interest in this paper are surfaces and
threefolds.  
There are many results we prove in this paper and not all can be listed in this abstract. Here are some of the results we prove inequalities related to the Bogomolov-Miyaoka-Yau inequality: in Corollary~\ref{surfaces-with-non-negative-h11-cor} that $c_1^2\leq \max(5c_2+6b_1,6c_2)$ holds for a large class of surfaces of general type. In Theorem~\ref{strange-thm} we prove that for a smooth, projective, Hodge-Witt, minimal surface of general type  (with additional assumptions such as slopes of Frobenius on $H^2_{cris}(X)$ are $\geq\frac{1}{2}$) that $$c_1^2\leq 5c_2.$$ We do not assume any lifting, and novelty of our method lies in our use of slopes of Frobenius and the slope spectral sequence.  We also construct new birational invariants of surfaces.  Applying our methods to threefolds, we characterize Calabi-Yau threefolds with $b_3=0$. We show that for any Calabi-Yau threefold $b_2\geq \frac{1}{2}c_3-1$ and that threefolds which lie on the line $b_2=\frac{1}{2}c_3-1$ are precisely those with $b_3=0$ and threefolds with $b_2=\frac{1}{2}c_3$ are characterized as Hodge-Witt rigid (included are rigid Calabi-Yau threefolds which have torsion-free crystalline cohomology and whose Hodge-de Rham spectral sequence degenerates).
\end{abstract}

 \maketitle
\epigraph{Torsten Ekedahl and Michel Raynaud}{In memoriam}
\markboth{Kirti Joshi}{Geography of surfaces I}
\tableofcontents

\section{Introduction}\label{introduction}
This is a modest attempt to study, in a systematic manner, the
structure of low dimensional varieties using $p$-adic invariants.
The main objects of interest in this paper are surfaces and
threefolds. It is known that there are many examples of unexpected behavior, including the failure of the famous Bogomolov-Miyaoka-Yau inequality, 
in surfaces and even exotic behavior in threefolds. Our focus  here is on obtaining general positive results.

We will take as thoroughly understood, the theory of algebraic surfaces over complex numbers. In
Section~\ref{notations} we recall results in crystalline
cohomology, and the theory of de Rham Witt complex which we use in
this paper. Reader familiar with \cite{illusie79b},
\cite{illusie83b}, \cite{ekedahl84}, \cite{ekedahl85} and
\cite{ekedahl-diagonal} is strongly advised to skip this section.

Here are some questions which this paper attempts to
address.

\subsection{The problem of Chern number inequalities for surfaces}
In Section~\ref{hodge-witt-of-surfaces} we begin with one of the
main themes of the paper. Let $X/k$ be a smooth, projective surface over an algebraically closed field $k$, of characteristic $p>0$.  There has been considerable work, of great depth and beauty, on Bogomolov-Miyaoka-Yau inequality in positive characteristic (as well as in characteristic zero).  Despite this, our understanding of the issue remains, at the best, rather primitive.  It is has been known for some time that the famous inequality of Bogomolov-Miyaoka-Yau  (see \cite{bogomolov78}, \cite{miyaoka77}, \cite{yau77}):
\be
c_1^2\leq 3c_2
\ee for Chern numbers of $X$  fails in general.
In fact it is possible to give examples of surfaces of general type in characteristic $p\ge5$ such that
\be
c_1^2>pc_2,
\ee
and even surfaces for which
\be
c_1^2>p^nc_2,
\ee
for suitable $n\geq 1$.

The following questions arise at this point:
\begin{enumerate}
\item What is the weakest inequality for Chern numbers which holds for a large class of surfaces?
\item Is there a class of surfaces for which one can prove some inequality of the form
$c_1^2\leq Ac_2$ (with $A>0$)?
\item Where does the obstruction to  Chern class inequality for $X$ originate?
\end{enumerate}
The questions are certainly quite vague, but also quite difficult: as the examples of surfaces of general type with $c_1^2>pc_2$ illustrates, and are not new (for example the first was proposed in \cite{shepherd-barron91b}).  In this paper we provide some answers to all of these questions. Our answers are not satisfactory, at least to us, but should serve as starting point for future investigations and to pose more precise questions.

Much of the existent work on this subject has been carried out from the geometric point view. However as we point out in this paper, the problem is not only purely of \textit{geometric} origin, but rather of  \textit{arithmetic} origin. More precisely, \textit{infinite $p$-torsion} of the \textit{slope spectral sequence} intervenes in a crucial way, already, in any attempt to prove the weaker inequality first considered in \cite{deven76a}:
\be
c_1^2\leq 5c_2.
\ee
Before proceeding further, let us dispel the notion that infinite torsion in the slope spectral sequence is in any way ``pathological.'' Indeed any supersingular K3 surface; any abelian variety of dimension $n$ and of $p$-rank at most $n-2$ and  any product of smooth, projective curves with supersingular Jacobians, and even Fermat varieties of large degree (for $p$ satisfying a suitable congruence modulo the degree) all have infinite torsion in the slope spectral sequence (this list is by no means exhaustive or complete) and as is well-known, this class of varieties is quite reasonable from every other geometric and cohomological point of view. So we must view the presence of infinite torsion  as the entry of the \textit{subtler arithmetic of the slope spectral sequence} of the variety \textit{into the question of its geometry} rather than as a manifestation of any pathological behavior.

Consider the first question. In Section~\ref{chern-inequalities} we take up the topic of Chern
class inequalities of the type $c_1^2\leq 5c_2$ and $c_1^2\leq
5c_2+6b_1$ (such inequalities were first considered in \cite{deven76a}).
We show in Theorem~\ref{surfaces-with-non-negative-h11} that
\bes c_1^2\leq 5c_2+6b_1
\ees
for a large class of surfaces. In particular one has for such surfaces (Corollary~\ref{surfaces-with-non-negative-h11-cor}):
\bes c_1^2\leq \max(5c_2+6b_1,6c_2).
\ees
and  if $b_1=0$ then 
$c_1^2\leq 5c_2$. Surfaces covered in Theorem~\ref{surfaces-with-non-negative-h11} includes surfaces  whose Hodge-de Rham spectral sequence degenerates at $E_1$ or which lift to $W_2$ and have torsion free crystalline cohomology (i.e. Mazur-Ogus surfaces), or surfaces which are ordinary or more generally Hodge-Witt.  
\textit{Moreover there also exists surfaces which do not satisfy these inequalities.} 

Readers familiar with  geography of surfaces over $\C$ will recall that geography of surfaces (see \cite{persson85}) over $\C$ is planar--with the numbers $c_1^2,c_2$ serving as variables in a plane and Bogomolov-Miyaoka-Yau line $c_1^2=3c_2$ representing the absolute boundary beyond which no surfaces of general type can live. On the other hand we show that the geography of surfaces in positive characteristic is non-planar and three dimensional, involving variables $c_1^2,c_2,b_1$  with the plane $c_1^2= 5c_2+6b_1$ serving as a natural boundary and unlike the classical case  of surfaces over $\C$, the region $c_1^2>5c_2+6b_1$  is also populated. In our view surfaces which live in the region $c_1^2\leq 5c_2+6b_1$ are the ones which can hope to understand.

Proposition~\ref{chern-inequalities2} provides the following important criterion:
  $$c_1^2\leq 5c_2\iff m^{1,1}-2T^{0,2}\geq b_1,$$
here the term $T^{0,2}$  is of de Rham-Witt origin and is a measure of infinite torsion in the slope spectral sequence, and the term $m^{1,1}$ is of crystalline origin involving slopes of Frobenius in the second crystalline cohomology of $X$. This criterion translates the problem of Chern number  inequality to an inequality involving slopes of Frobenius, de Rham-Witt contributions and is the reason for our contention that the problem of the Bogomolov-Miyaoka-Yau inequality in positive characteristic is intimately related to understanding the influence of the infinite torsion in the slope spectral sequence. 

But the  inequality $m^{1,1}-2T^{0,2}\geq b_1$  is still difficult to deal with in practice: the $m^{1,1}$ term and the $b_1$ terms live in two different cohomologies ($H^1_{cris}$ and $H^2_{cris}$ resp.) and these two cohomologies are not as correlated as they are in characteristic zero (the proof of \cite{deven76a} can be viewed as establishing a correlation between $H^1_{dR}$ and $H^2_{dR}$ in characteristic zero). 

The next step is to resolve this difficulty. This is carried out in Proposition~\ref{m11andpg}  (resp. Proposition~\ref{m11andpg2}) for Hodge-Witt surfaces (resp. for Mazur-Ogus surfaces), which show that 
$$m^{1,1}\geq 2p_g\implies c_1^2\leq 5c_2\ \ (\text{resp. }m^{1,1}\geq 4p_g\implies c_1^2\leq 5c_2).$$
(Here $m^{1,1}$ is the slope
number (see \ref{slope-number-definition} for the definition) and
$p_g$ is the geometric genus.) The criterion of Proposition~\ref{m11andpg} is used in the proof of one of the main results of this paper: Theorem~\ref{strange-thm} which we discuss next.

\begin{theorem*}
 If $X/k$ is a smooth, projective, minimal surface of general type over a perfect field of characteristic $p>0$ which satisfies the following conditions:
 \begin{enumerate}
 	\item $c_2>0$,
 	\item $p_g>0$,
 	\item $X$ is Hodge-Witt
 	\item $\Pic(X)$ is reduced or $H^2_{cris}(X/W)$ is torsion-free,
 	\item and $H^2_{cris}(X/W)$ has no slope $<\frac{1}{2}$.
 \end{enumerate}
 Then
 $$
 c_1^2\leq 5c_2.
 $$	
\end{theorem*}

\emph{The novelty of our method lies in our use of slopes of Frobenius to prove such an inequality (when de Rham-Witt torsion is controlled--by the Hodge-Witt hypothesis).} It seems to us that this is certainly not the most optimal result which can be obtained by our methods, but should serve as a starting point for understanding Bogomolov-Miyaoka-Yau type inequality for surfaces using slopes of Frobenius.  
For a detailed discussion of our hypothesis and simple examples of surfaces of general type which satisfy all of the above hypothesis reader is referred to the remarks preceding and following the theorem, but here we point out that in Theorem~\ref{th:lower-bound-for-slopes} we show that surfaces which have no slope zero part in $H^2_{cris}$ have slopes  bounded from below by $\frac{1}{p_g+1}$. At any rate we demonstrate that the class of surfaces which satisfy  assumptions (1)--(5) is a non-empty (in general) and form a  locally closed subset in moduli if we assume fixed Chern classes. 

Let us also point out that after the 2012 version of this paper was circulated on the preprint arxiv, I was informed by Adrian Langer that he has obtained a proof of Bogomolov-Miyaoka-Yau inequality for surfaces lifting to $W_2$ (see \cite{langer15}). This is certainly an important development in the subject. However note that in Theorem~\ref{strange-thm} we  do not assume any lifting hypothesis and our methods have no overlap with those of \cite{langer15}.

Let us remark that for  \textit{ordinary surfaces}   $m^{1,1}-2T^{0,2}\geq b_1$, as $T^{0,2}=0$ by a Theorem of \cite{illusie83b}, one reduces to the inequality
\bes
m^{1,1}\geq b_1
\ees
and this is still non-trivial to prove (and we do not know how to prove it). If we assume in addition to ordinarity that $X$ has torsion-free crystalline cohomology, then $m^{1,1}=h^{1,1}$ and hence the inequality is purely classical--involving Hodge and Betti numbers. But the proof of the inequality in this case would, nevertheless be non-classical. In this sense the ordinary case is closest to the classical case, but it appears to us that this case nevertheless lies beyond classical geometric methods. 
In Subsection~\ref{recurringfantasy} we describe a \textit{recurring fantasy} to prove $$c_1^2\leq 5c_2+6$$ for ordinary surfaces (this was not included in any  version prior to the 2012 version of this paper) and also $$c_1^2\leq 6c_2$$ for all ordinary surfaces except for those with $0\leq c_1^2<36,0\leq c_2<6$ (which form a bounded family at any rate). This recurring fantasy should be considered the \textit{de Rham-Witt avatar of van de Ven's Theorem} \cite{deven76a}.

Surfaces which do not satisfy $c_1^2\leq 5c_2+6b_1$ are
particularly extreme cases of failure of the
Bogomolov-Miyaoka-Yau inequality and their properties are studied in Theorem~\ref{negative-hw11} and the remarks following it. They all exhibit following properties:  are not Hodge-Witt, exhibit
non-degeneration of Hodge de Rham or presence of crystalline
torsion, and if $c_2>0$ then $\Omega^1_X$ is unstable. 

The key tool in these and other results proved in this paper are certain invariants of non-classical nature, called Hodge-Witt numbers, which were
introduced by T.~Ekedahl (see \cite{ekedahl-diagonal}) and use a
remarkable formula of R.~Crew (see loc. cit.) and in particular for surfaces the key tool is
the Hodge-Witt number $h^{1,1}_W$ of surfaces.  This integer can be negative (and its negativity signals failure of Chern class inequalities) and in fact it can be arbitrarily negative. In Proposition~\ref{non-negative-hw-for-surfaces} we
note that for a smooth, projective surface $h^{1,1}_W$ is the only
one which can be negative, the rest of $h^{i,j}$ are non-negative.  In
Subsection~\ref{enriques} we use the Enriques classification of
surfaces to prove that  if
$h^{1,1}_W$ is negative then $X$ is of general type or quasi-elliptic (these occur only for $p=2,3$). In Subsection~\ref{lower-bounds} we investigate lower bounds on
$h_W^{1,1}$. For instance we note that if $X$ is of general type
then $-c_1^2\leq h_W^{1,1}\leq h^{1,1}$, except possibly for
$p\leq 7$ and $X$ is fibered over a curve of genus at least two and
the generic fiber is a singular rational curve of arithmetic genus at most
four. It seems rather optimistic to conjecture that if $b_1\neq 0$ and $h^{1,1}_W<0$ then
$X\to \Alb(X)$ has one dimensional image (and so such surfaces
admit a fibration with an irrational base).

\subsection{A Chern inequality for Calabi-Yau Threefolds}
In Section~\ref{hodge-witt-numbers-of-threefolds} we take up the
study of Hodge-Witt numbers of smooth projective threefolds.
Applying our methods to Calabi-Yau threefolds, We show that up to symmetry, the only possibly negative Hodge-Witt number is $h^{1,2}_W$ and
in
Theorem~\ref{non-liftable-b3} we obtain  a
complete characterization (in any positive characteristic) of
Calabi-Yau threefolds with $h^{1,2}_W<0$.  These are
are precisely the threefolds for which the Betti number $b_3=0$. 
As a   corollary we deduce (in Corollary~\ref{cor:calabi-yau-h12-lower-bound}) that  if $X$ is any smooth, projective Calabi-Yau threefold then $h^{1,2}_W\geq -1$, which is equivalent to the inequality (valid for all projective Calabi-Yau threefolds in positive characteristic):
\bes
b_2\geq \frac{1}{2}c_3-1.
\ees
On the other hand the condition  $h^{1,2}_W\geq 0$ (for Calabi-Yau threefolds) is equivalent to the geometric inequality (always valid in characteristic zero):
\bes
b_2\geq \frac{1}{2}c_3,
\ees
In particular such threefolds which do not satisfy this inequality  cannot lift to characteristic zero.
Most known examples of non-liftable Calabi-Yau threefolds have this
property ($b_3=0$) and it is quite likely, at least if $H^2_{cris}(X/W)$ is torsion-free, that all non-liftable Calabi-Yau
threefolds are of this sort. 
In Section~\ref{Calabi-Yau-Geography} we establish some results in geography of Calabi-Yau threefolds. In particular we note that unlike characteristic zero, in characteristic $p>0$ there is a new region in the geography of such threefolds  which is unconnected from the realm of classical Calabi-Yau threefolds and many non-liftable Calabi-Yau threefolds live on this island(see the section for the definition).

We expect that if a Calabi-Yau threefold has non-negative Hodge-Witt numbers and torsion free crystalline cohomology then it lifts to $W_2$ (for $p\geq 5$). 
Our conjecture is that if $H^2_{cris}(X/W)$ is torsion-free, equivalently if $H^0(X,\Omega^1_X)=0$, for a Calabi-Yau threefold, and if this inequality holds then $X$ lifts to $W_2$ (for $p\geq 5$). In Section~\ref{sec:yobuko-theorem} we describe our formulation of a remarkable recent  result of  F.~Yobuko (see \cite{yobuko16}) which provides some evidence towards our conjecture. In Proposition~\ref{pro:hodge-witt-rigidity} we show that if $X$ is Mazur-Ogus then $$c_3=2b_2\iff h^{1,2}=0 \iff X \text{is a rigid Calabi-Yau threefold}.$$

\subsection{Other results proved in the paper}
Here are some of the other results we prove in this paper. 
\subsubsection{Computing Domino numbers of Mazur-Ogus varieties}
We prove (see
Theorem~\ref{domino-numbers-for-mazur-ogus}) that the domino
numbers of a smooth, projective variety whose Hodge de Rham
 spectral sequence degenerates and whose crystalline cohomology is
 torsion free are completely determined by the Hodge numbers and
 the slope numbers (in other words they are completely determined
 by the Hodge numbers and the slopes of Frobenius). This had been
 previously proved by Ekedahl~\cite{illusie90a} for abelian
 varieties. In Subsection~\ref{explicitformulas} we compute domino numbers of smooth hypersurfaces in $\P^n$ (for $n\leq 4$).
 
 \subsubsection{Birational invariance of Domino numbers of smooth, projective surfaces}
 The other question of interest, especially for surfaces is: how do
 the various $p$-adic invariants reflect in the Enriques-Kodaira
 classification? Of course, the behavior of cohomological
 invariants is quite well-understood. But in positive
 characteristic there are other invariants which are of a
 non-classical nature. These new invariants of surfaces are
 at moment defined only for smooth, projective surfaces. But they are of
 birational nature in the sense that: two smooth, projective
 surfaces which are birational surfaces have the same invariants.
 These are: the $V$-torsion, the N\'eron-Severi torsion, the exotic
 torsion and in Proposition~\ref{birational-invariance-of-domino} we prove that if
$X'$ and $X$ are two smooth surfaces which are birational then
 the dominos associated
 to the corresponding differentials 
 \bes
 H^2(X,W(\O_X))\to H^2(X,W\Omega^1_X)
 \ees
 and 
 \bes
 H^2(X',W(\O_{X'}))\to H^2(X',W\Omega^1_{X'})
 \ees
are naturally isomorphic.  In particular the dimension of the domino, denoted 
 by $T^{0,2}(X)$, is a new birational numerical invariant of smooth, projective
 surfaces which lives only in positive characteristic. We also study torsion in $H^2_{cris}(X/W)$ in terms of
 the Enriques classification, making precise several results found
 in the existing literature.
 
\subsubsection{Crystalline torsion and a question of V.~B.~Mehta}
In Section~\ref{geography-of-crystalline-torsion} we digress a
little from our main themes. This section may be well-known to
the experts. We consider torsion in the second crystalline
cohomology of $X$. It is well-known that torsion in the second
cohomology of a smooth projective surface is a birational
invariant, and in fact the following variant of this is true in
positive characteristic (see
Proposition~\ref{torsion-species-and-blowups}): torsion of every
species (i.e. N\'eron-Severi, the $V$-torsion, and the exotic
torsion) is a birational invariant. We also note that surfaces of
Kodaira dimension at most zero do not have exotic torsion. The
section ends with a criterion for absence of exotic torsion which
is often useful in practice. 

In Theorem~\ref{precise-torsion} we show that if $X$ is a smooth,
projective surface of general type with exotic torsion then $X$
 has Kodaira dimension $\kappa(X)\geq 1$. A surface with
 $V$-torsion must either have $\kappa(X)\geq 1$ or $\kappa(X)=0$,
 $b_2=2$ and $p_g=1$ or $p=2$, $b_2=10$ and $p_g=1$. This theorem
 is proved via
 Proposition~\ref{classification-of-torsion-for-kappa0} where we
 describe torsion in $H^2_{cris}(X/W)$ for surfaces with Kodaira
 dimension zero.

In Section~\ref{mehta-question} 
we answer a question of Mehta (for surfaces)  about crystalline torsion. We show that
any smooth, projective surface  $X$ of Kodaira dimension at most
zero has a Galois \'etale cover $X'\to X$ such that
$H^2_{cris}(X'/W)$ is torsion free. Thus crystalline torsion in
these situations, can in some sense, be uniformized, or
controlled. We do not know if this result should be true without
the assumption on Kodaira dimension.

In \cite{joshi05b} which is a thematic sequel (if we get around to completing it) to this paper we will study the properties of a refined Artin invariant of families of Mazur-Ogus surfaces and related stratifications. While bulk of this paper was written more than a decade ago, and after a few rejections from journals where I thought (rather naively) that the paper could (or perhaps should) appear, I lost interest in its publication and the paper has gestated at least since 2007. Over time there have been many additions which I have made to this paper (and I have rewritten the introduction to reflect my current thinking on this matter), but theme of the paper remains unchanged. Notable additions are: Theorem~\ref{strange-thm} and Propositions~\ref{m11andpg} and Proposition~\ref{m11andpg2} which are of later vintage (being proved around 2012-13); and Subsection~\ref{ordinaryconj} has matured for many years now, but has been added more recently; I also added a computation of Domino numbers of hypersurfaces Subsection~\ref{explicitformulas} around 2009. In 2017 during the course of revisions suggested by the referee I also proved and added Theorem~\ref{th:lower-bound-for-slopes} (which sheds some light on the hypothesis of Theorem~\ref{strange-thm}), Theorem~\ref{th:hodge-newton-polygons}. Theorem~\ref{th:yobuko-reformulation} was added upon reading F.~Yobuko's preprint which also appeared in 2017.

\subsection{Acknowledgements}
Untimely death of Torsten Ekedahl in 2011 reminded me of this manuscript again and I decided to revive it from its slumber (though it takes a while to wake up after such a  long slumber). Unfortunately Michel Raynaud also passed away while this paper was being revised for publication.   

\emph{I dedicate this paper to the memory of Torsten Ekedahl and Michel Raynaud. I  never had the opportunity to meet Ekedahl, but his work has been a source of inspiration for a long time. The few meetings I had with Raynaud during my visits to Orsay, I recall with great pleasure.}

This paper clearly owes its existence to the work of Richard Crew, Torsten Ekedahl, Luc Illusie and Michel Raynaud.  I take this opportunity to thank Luc Illusie and Michel Raynaud for encouragement.  Thanks  are also due to Minhyong Kim for constant encouragement while early versions of this paper were being written (he was at Arizona at the time the paper was written). I would like to thank the Korea Institute of Advanced Study and especially thank the organizers of the
International Workshop on Arithmetic Geometry in the fall of 2001,
where some of the early results were announced, for support.

I am also deeply indebted to the referee for a number of suggestions and corrections which have vastly improved this manuscript and also to the Editors of this Journal, especially Fedor Bogomolov, for their patience in face of the long delay in revising this manuscript.

I have tried to keep this paper as self-contained as possible. This may leave the casual reader the feeling that these results are elementary, and to a certain extent they are; but we caution the reader that this feeling is ultimately illusory as we use  deep work  of Crew, Ekedahl, Illusie and Raynaud  on the slope spectral sequence which runs into more than  four hundred pages of rather profound and beautiful mathematics.

\section{Notations and Preliminaries}\label{notations}
\subsection{Witt vectors} Let $p$ be a prime number and let $k$ be a perfect field
of characteristic $p$. Let $\bk$ be an algebraic closure of $k$. Let
$W=W(k)$ be the ring of Witt vectors of $k$ and let $W_n=W/p^n$ be the
ring of Witt vectors of $k$ of length $n\geq 1$. Let $K$ be the
quotient field of $W$. Let $\sigma$ be the Frobenius morphism
$x\mapsto x^p$ of $k$ and let $\sigma:W\to W$ be its canonical lift to
$W$. We will also write $\sigma:K\to K$ for the extension of
$\sigma:W\to W$ to $K$.
\subsection{The Cartier-Dieudonne-Raynaud Algebra}
Following \cite[page~90]{illusie83b} we write $R$ for the Cartier-Dieudonne-Raynaud algebra or more simply the
Raynaud algebra of $k$. Recall that $R$ is a
$W$-algebra generated by symbols $F,V,d$ with the following relations:
\begin{eqnarray*}
FV&=&p\\
VF&=&p\\
{\rm and}\ \forall\,a\in W, \ Fa&=&\sigma(a)F\\
aV&=&V\sigma^{-1}(a)\\
d^2&=&0\\
FdV&=&d\\
da&=&ad.
\end{eqnarray*}
The Raynaud algebra is graded $R=R^0\oplus R^1$ where $R^0$ is the
$W$-subalgebra generated by symbols $F,V$ with relations above and
$R^1$ is generated as an $R^0$ bi-module by $d$ (see
\cite[page~90]{illusie83b}).
\subsection{Explicit description of $R$}
Every element of $R$ can be written uniquely as a sum
\begin{equation}
\sum_{n>0}a_{-n}V^n+
\sum_{n\geq0}a_nF^n+\sum_{n>0}b_{-n}dV^n+\sum_{n\geq0}b_nF^nd
\end{equation}
where $a_n,b_n\in W$ for all $n\in\Z$ (see \cite[page
90]{illusie83b}).

\subsection{Graded modules over $R$}\label{complexes}
        Any graded $R$-module $M$ can be thought of as a complex
$M=M^{\sbull}$ where $M^i$ for $i\in\Z$ are $R^0$ modules and the
differential $M^i\to M^{i+1}$ is given by $d$ with $FdV=d$ (see
\cite[page 90]{illusie83b}).
\emph{From now on we will assume that all  $R$-modules are   graded}.

\subsection{Canonical Filtration}\label{canonical-filtration}
On any $R$-module $M$ we define a filtration (see \cite[page
92]{illusie83b} by
\begin{eqnarray}
\Fil^nM&=&V^nM+dV^nM\\
\gr^nM&=&\Fil^nM/\Fil^{n+1}M
\end{eqnarray}
In particular we set $R_n=R/\Fil^nR$.
\subsection{Topology on $R$}\label{topology}
We topologize an $R$ module $M$ by the linear topology given by
$\Fil^nM$ (see \cite[page 92]{illusie83b}).
\subsection{Complete modules}\label{complete-modules}
We write $\hat{M}=\liminv_{n}M/\Fil^nM$ and call $\hat{M}$ the
completion of $M$, and say $M$ is complete if $\hat{M}=M$. Note that
$\hat{M}$ is complete and one has $(\hat{M})^i=\liminv_{n}
M^i/\Fil^nM^i$ \cite[section 1.3, page 90]{illusie83b}.
\subsection{Differential}\label{differential0}
Let $M$ be an $R$-module, for all $i\in\Z$ we let
\begin{eqnarray}
Z^iM&=&\ker(d:M^i\to M^{i+1})\\
B^iM&=&d(M^{i-1})
\end{eqnarray}
The $W$-module $Z^iM$ is stable by $F$ but not by $V$ in general and
we let
\begin{equation}
V^{-\infty}Z^i=\bigcap_{r\geq 0}V^{-r}Z^i
\end{equation}
where $V^{-r}Z^i=\{x\in M^i| V^r(x)\in Z^i\}$. Then $V^{-\infty}Z^i$
is the largest $R^0$ submodule of $Z^i$ and $M^i/V^{-\infty}Z^i$ has
no $V$-torsion (see \cite[page 93]{illusie83b}).

The $W$-module $B^i$ is stable by $V$ but not in general by $F$.
We let $$F^\infty B^i=\bigcup_{s\geq 0}F^sB^i.$$ Then $F^\infty
B^i$ is the smallest $R^0$-submodule of $M^i$ which contains $B^i$
(see \cite[page 93]{illusie83b}).

\subsection{Canonical Factorization}\label{differential1}
The differential $M^{i-1}\to M^i$ factors canonically as
\begin{equation}\label{domino-factorization}
M^{i-1}\to M^{i-1}/V^{-\infty}Z^{i-1}\to F^\infty B^i \to M^i
\end{equation}
(see \cite[page 93, 1.4.5]{illusie83b}) we write
$\tilde{H}^i(M)=V^{-\infty}Z^i/F^\infty B^i$, is called the \emph{heart}
of the differential $d:M^{i}\to M^{i+1}$ and we will say that the
differential is {\em heartless} if its heart is zero.

\subsection{Profinite modules}\label{profinite-modules}
A (graded) $R$-module $M$ is \emph{profinite} if $M$ is complete and for
all $n,i$ the $W$-module $M^i/\Fil^nM^i$ is  of finite length
\cite[Definition 2.1, page 97]{illusie83b}.

\subsection{Coherence}\label{coherence}
A (graded) $R$-module $M$ is \emph{coherent} if it is of bounded degree,
profinite and all the hearts $\tilde{H}^i(M)$ are of finite type over $W$
\cite[Theorem~3.8, 3.9, page 118]{illusie83b}.

\subsection{Dominoes} An $R$-module $M$ is a \textit{domino} if $M$ is graded, profinite, concentrated in
two degrees (say) $0,1$ and $V^{-\infty}Z^0=0$ and $F^\infty
B^1=M^1$. 

If $M$ is any graded, profinite $R$-module then the canonical
factorization  of $d:M^{i}\to M^{i+1}$, given in
\ref{domino-factorization}, gives a domino: $M^i/V^{-\infty}Z^i\to
F^\infty B^{i+1}$ (see \cite[2.16, page 110]{illusie83b}),
which we call the \emph{domino associated to the differential}
$d:M^i\to M^{i+1}$.

\subsection{Dimension of a Domino} Let $M$ be a domino, then we define $T(M)=\dim_k
M^0/VM^0$ and call it the \emph{dimension of the domino}. If $M$ is any
$R$-module we write $T^i(M)$ for the dimension of the domino
associated to the differential $M^i\to M^{i+1}$. It is standard
that $T^i(M)$ is finite \cite[Proposition~2.18, page
110]{illusie83b}.

\subsection{Filtration on dominoes}
Any domino $M$ comes equipped with a finite decreasing filtration by
$R$-submodules such that the graded pieces are certain standard one
dimensional dominos $U_j$ (see
\cite[Proposition~2.18, page 110]{illusie83b}). Note that the filtration may not be unique but the number of factors is the same for all such filtrations.

\subsection{One dimensional dominoes}
        The $U_j$, one for each $j\in\Z$, provide a complete list of
all the one dimensional dominos (see \cite[Proposition~2.19, page
111]{illusie83b}).

\subsection{Domino devissage lemmas} Further by
\cite[Lemma~4.2, page 12]{ekedahl-diagonal} one has
\begin{equation}
\Hom_R(U_i,U_j)=0
\end{equation}
 if $i>j$, and
\begin{equation}
\Hom_R(U_i,U_i)=k.
\end{equation}

\subsection{Geometrically connectedness assumption}
To avoid tedious repetition  we will assume throughout this paper 
that all schemes which appear in this paper are geometrically connected. \emph{We caution the reader that this assumption may not always appear in the statements of the theorems but is tacitly used in many of the proofs.} (We thank the referee for reminding us of this).

\subsection{Slope spectral sequence}

Let $X$ be a smooth, projective scheme over $k$, we will write $H^*_{\cris}(X/W)$ for
the crystalline cohomology of $X$. If $X$ is
smooth and proper, in \cite{illusie79b}, one finds the
construction of the de Rham-Witt complex $W\Omega^\mydot_X$. The
construction of this complex is functorial in $X$. This complex
computes $H^*_{cris}(X/W)$ and one has a spectral sequence (the
\emph{slope spectral sequence} of $X$)
$$E_1^{i,j}=H^j(X,W\Omega^i_X)\Rightarrow H^{i+j}_{cris}(X/W).$$
The construction slope spectral sequence is also functorial in $X$
and for each $j\geq 0$,  $H^j(X,W\Omega^\mydot_X)$ is a graded, profinite and coherent
$R$-module. \emph{Modulo torsion}, the slope spectral sequence always
degenerates at $E_1$. 



\subsection{Classical invariants of surfaces}\label{torsion}
Let $X/k$ be a smooth projective surface. Then recall the
following standard notation for numerical invariants of $X$. We
will write $b_i=\dim_{\Q_\ell} H^i_{et}(X,\Q_\ell)$,
$q=\dim\Alb(X)=\dim\Pic^0(X)_{\rm red}$; $2q=b_1$, $h^{ij}=\dim_k
H^j(X,\Omega_X^i)$; and $p_g(X)=h^{0,2}=h^{2,0}$. Then one has the
following form of the Noether's formula:
\begin{equation}\label{noether-formula1}
10+12p_g=c_1^2+b_2+8q+2(h^{0,1}-q),
\end{equation}
one has by definition $c_1^2=K_X^2$, the self intersection of the canonical bundle of $X$. The point is that all the terms are non-negative except
possibly $c_1^2$. Further by \cite[page 25]{bombieri76} we have
\begin{equation}
0\leq h^{0,1}-q\leq p_g.
\end{equation}
Formula \eqref{noether-formula1} is easily seen to be equivalent to
the usual form of Noether's formula
\begin{equation}\label{noether-formula2}
12\chi(\O_X)=c_1^2+c_2=c_1^2+\chi_{et}(X).
\end{equation}

\subsection{Hodge-Noether formula} In addition, when the ground field $k=\C$, we get yet
another form of Noether's formula which is a consequence of
\eqref{noether-formula2} and the Hodge decomposition:
\begin{equation}\label{noether-formula3}
\begin{aligned}
h^{1,1} & =10\chi(\O_X)-c_1^2+b_1,  \textit{ or equivalently, }\\
  h^{1,1} &=\frac{5c_2-c_1^2}{6}+b_1.
\end{aligned}
\end{equation}
We will call this the \emph{Hodge-Noether formula}.

\subsection{Hodge-Witt invariants and other invariants}
\label{role-of-hodge-witt} In the next few subsections we recall
results on Hodge-Witt numbers \cite[page 85]{ekedahl-diagonal}  of
surfaces and threefolds. We recall the definition of Hodge-Witt
numbers and their basic properties.

\subsection{Slope numbers}\label{slope-number-definition} Let $X$ be a smooth projective
variety over a perfect field $k$. The \emph{slope numbers} of $X$ are
defined by (see \cite[page 85]{ekedahl-diagonal}):
\begin{eqnarray*}
m^{i,j}&=&\sum_{\lambda\in [i-1,i)}(\lambda-i+1)\dim_K
H^{i+j}_{\cris}(X/W)_{[\lambda]} \\
& &\qquad+\sum_{\lambda\in
[i,i+1)}(i+1-\lambda)\dim_K H^{i+j}_{\cris}(X/W)_{[\lambda]}
\end{eqnarray*}
where  the summation is over all the slopes of Frobenius
$\lambda$ in the indicated intervals and $H^{i+j}_{\cris}(X/W)_{[\lambda]}$ denotes the slope $\lambda$ part of $H^{i+j}_{\cris}(X/W)\tensor_WK$. 
Let me note that  slope numbers of a smooth, projective variety are non-negative integers (see \cite{crew85}).

\subsection{Dominoes in the slope spectral sequence}
\label{domino-number-definition} Let $X$ be a smooth projective
variety. For all $i,j\geq 0$ the differential 
\be d:H^j(X,W\Omega^{i}_X)\to
H^j(X,W\Omega_X^{i+1}).
\ee
admits a canonical factorization 
$$\xymatrix{ H^j(X,W\Omega^{i}_X)\ar@{>>}[d] \ar[r]^d & H^j(X,W\Omega_X^{i+1})\\
	\frac{H^j(X,W\Omega^{i}_X)}{V^{-\infty}Z^{i}}\ar [r] & F^\infty B^{i+1}\ar@{^{(}->}[u].
}
$$
where the lower arrow is a domino, called \emph{the domino  associated with the differential}
\be d:H^j(X,W\Omega^{i}_X)\to
H^j(X,W\Omega_X^{i+1}).
\ee
and is denoted by $\dom^{i,j}(H^{\sbull}(X,W\Omega_X^{\sbull}))$.

The \emph{domino numbers} $T^{i,j}$ of $X$ are defined by
(see \cite[page 85]{ekedahl-diagonal}):
\begin{equation*}
T^{i,j}=\dim_k \dom^{i,j}(H^{\sbull}(X,W\Omega_X^{\sbull}))
\end{equation*}
        in other words, $T^{i,j}$ is the dimension of the domino
associated to the differential 
\be d:H^j(X,W\Omega^{i}_X)\to
H^j(X,W\Omega_X^{i+1}).
\ee
Note that these are non-negative integers.

\subsection{Hodge-Witt Numbers}
\label{hodge-witt-definition} Let $X$ be a smooth projective
variety over a perfect field $k$. The \emph{Hodge-Witt numbers} of $X$
are defined by the formula (see \cite[page 85]{ekedahl-diagonal}):
\begin{equation*}
h^{i,j}_W=m^{i,j}+T^{i,j}-2T^{i-1,j+1}+T^{i-2,j+2}.
\end{equation*}
In particular one sees from the definition that Hodge-Witt numbers are also integers.

\subsection{Formulaire}\label{formulaire}
The Hodge-Witt numbers, domino numbers and the slope numbers
satisfy the following properties which we will now list. See
\cite{ekedahl-diagonal} for details.

\subsubsection{Slope number symmetry}
        For all $i,j$ one has the symmetries (see \cite[Lemma 3.1,
page 112]{ekedahl-diagonal}):
\begin{equation*}\label{slope-symmetry}
        m^{i,j}=m^{j,i}=m^{n-i,n-j},
\end{equation*}
here the first is a consequence of Hard Lefschetz Theorem and the
second is a consequence of Poincar\'e duality.  Further these
numbers are obviously non-negative:
\begin{equation*}
m^{i,j}\geq0.
\end{equation*}
\subsubsection{Theorems of Ekedahl}\label{ekedahl-theorems}
The following formulae give relations to Betti numbers (see
\cite[Theorem 3.2, page 85]{ekedahl-diagonal}):
\begin{equation*}\label{slope-betti}
\sum_{i+j=n}m^{i,j}=\sum_{i+j=n}h^{i,j}_W=b_n
\end{equation*}
and one has Ekedahl's upper bound (see \cite[Theorem 3.2, page
86]{ekedahl-diagonal}):
\begin{equation*}\label{ekedahl-bound}
h^{i,j}_W\leq h^{i,j},
\end{equation*}
and  if $X$ is Mazur-Ogus (see
Subsection~\ref{mazur-ogus-varieties} for the definition of a
Mazur-Ogus variety) then we have
\begin{equation*}\label{ekedahl-hodge-witt-theorem}
h^{i,j}_W = h^{i,j}.
\end{equation*}

\subsubsection{Ekedahl's duality}
One has the following fundamental duality relation for dominos due to
Ekedahl (see \cite[Corollary~3.5.1, page 226]{ekedahl85}): for all
$i,j$, the domino $\dom^{i,j}$ is canonically dual to $\dom^{n-i-2,n-j+2}$
and in particular one has the equlity of the domino numbers
\begin{equation}\label{domino-duality}
T^{i,j}=T^{n-i-2,n-j+2}.
\end{equation}

\subsection{Crew's Formula} The Hodge and Hodge-Witt numbers of $X$ satisfy a
relation known as Crew's formula which we will use often in this
paper. The formula is the following
\begin{equation}\label{crew-formula}
\sum_j(-1)^jh^{i,j}_W=\chi(\Omega^i_X)=\sum_j(-1)^jh^{i,j}.
\end{equation}

\subsubsection{Hodge-Witt symmetry}\label{hodge-witt-symmetry}
For any smooth proper variety of dimension at most three we have for all $i,j$ (see \cite[Corollary 3.3(iii),page 113]{ekedahl-diagonal}):
\be 
	h^{i,j}_W=h^{j,i}_W.
\ee
Note this is proved in loc. cit. first for smooth projective varieties of dimension at most three (projectivity is  essential part as one makes  use of Deligne's Hard Lefschetz Theorem \cite[Th\'eor\`eme 4.1.1]{deligne80}) and then Ekedahl appeals to  resolution of singularities  (known for $p\geq 5$ for threefolds) to deduce the proper case. The argument of loc. cit. also works if we replace the use of resolution of singularities by de Jong's theorem on alterations.

\subsubsection{Hodge-Witt Duality}\label{hodge-witt-duality}
For any smooth projective variety of dimension $n$  we have
for all $i,j$ (see \cite[Corollary
3.2(i), page 113]{ekedahl-diagonal}):
\be 
	h^{i,j}_W=h^{n-i,n-j}_W.
\ee

\subsection{Crew's Formula for surfaces}
For surfaces, the formulas in \ref{crew-formula} take more
explicit forms (see \cite[page 85]{ekedahl-diagonal} and
\cite[page 64]{illusie83a}). We recall them now as they will play
a central role in our investigations. Let $X/k$ be a smooth
projective surface, $K_X$ be its canonical divisor, $T_X$ its
tangent bundle, and let $c_1^2,c_2$ be the usual Chern invariants
of $X$ (so $c_i=c_i(T_X)=(-1)^ic_i(\Omega^1_X)$). Then the
Hodge-Witt numbers of $X$ are related to the other numerical
invariants of $X$ by means of the following formulae \cite[page
114]{ekedahl-diagonal}:
\begin{eqnarray}\label{ekedahl-formula}
h^{0,1}_W&=&h^{1,0}_W\\
h^{0,1}_W&=&b_1/2\\
h^{0,2}_W&=&h^{2,0}_W\\
h^{0,2}_W&=&\chi(\O_X)-1+b_1/2\\
h^{1,1}_W&=&b_1+\frac{5}{6}c_2-\frac{1}{6}c_1^2
\end{eqnarray}
\subsection{Hodge-Witt-Noether Formula}
The formula for $h^{1,1}_W$ above and Noether's formula give the
following variant of the Hodge-Noether formula of
\eqref{noether-formula3}.  We will call this variant  the
\emph{Hodge-Witt-Noether} formula:
\begin{equation}\label{noether-formula4}
\begin{aligned}
h^{1,1}_W & =10\chi(\O_X)-c_1^2+b_1,  \textit{ or equivalently, }\\
h^{1,1}_W &=\frac{5c_2-c_1^2}{6}+b_1.
\end{aligned}
\end{equation}
This paper began with the realization that the above formula is the de Rham-Witt analogue of the  Hodge-Noether \eqref{noether-formula3} and will be central to our study of surfaces in this
paper.

\subsection{Hodge-Witt and Ordinary varieties}\label{la:hodge-witt-ordinary-def}
Let $X/k$ smooth, projective variety over a perfect field $k$ of characteristic $p>0$. 
Then
\benum
\item We say that $X$ is  \emph{Hodge-Witt} if  for each $i,j\geq 0$, $H^j(X,W\Omega^i_X)$ is a finite type module over $W$.
\item We say that $X$ is \emph{ordinary} if  $H^j(X,BW\Omega^i_X)=0$ for all $i,j\geq 0$. 
\eenum

Here are some examples (all of these assertions require proofs which are far from elementary). Any smooth, projective curve is Hodge-Witt.
An abelian variety is ordinary if and only if it has $p$-rank equal to its dimension. An abelian variety is Hodge-Witt if and only if its $p$-rank is at least as big as its dimension minus one. Any ordinary variety is a Hodge-Witt variety.  A K3 surface is Hodge-Witt if and only if it is of finite height and a K3 surface is ordinary if and only if its height is one.
In particular one sees that in all dimensions there exist Hodge-Witt varieties which are not ordinary.

\subsection{Mazur-Ogus and Deligne-Illusie
varieties}\label{mazur-ogus-varieties} In the next few subsections
we enumerate the properties of a class of varieties known as
Mazur-Ogus varieties. We will use this class of varieties at
several different points in this paper as well as its thematic
sequels so we elaborate some of the properties of this class of
varieties here.
\subsection{Mazur-Ogus varieties} A smooth, projective variety over a perfect field $k$
is said to be a Mazur-Ogus variety if it satisfies the following
conditions:
\begin{enumerate}
\item The Hodge de Rham spectral sequence of $X$ degenerates at
$E_1$, and
\item crystalline cohomology of $X$ is torsion free.
\end{enumerate}

\subsection{Deligne-Illusie varieties} A smooth, projective variety over a perfect field $k$
is said to be a Deligne-Illusie variety if it satisfies the
following conditions:
\begin{enumerate}
\item $X$ admits a flat lifting to $W_2(k)$ and
\item crystalline cohomology of $X$ is torsion free.
\end{enumerate}

\begin{remark}
\begin{enumerate}
\item The class of Mazur-Ogus varieties is quite reasonable for
many purposes and is rich enough to contain varieties with many de
Rham-Witt torsion phenomena. For instance any K3 surface is
Mazur-Ogus (in particular the supersingular K3 surface is
Mazur-Ogus). \item We caution the reader that our definition of
Deligne-Illusie varieties is  more restrictive than that conceived
by Deligne-Illusie. Nevertheless, we have the following
restatement of \cite{deligne87}.
\end{enumerate}
\end{remark}

\begin{theorem}\label{deligne-illusie-thm} If $p>\dim(X)$ then any Deligne-Illusie variety
$X$ is a Mazur-Ogus variety.
\end{theorem}

\begin{remark} It seems reasonable to expect that the inclusion of the
class of Deligne-Illusie varieties in the class of Mazur-Ogus
varieties is strict. However we do not know of an example. The
class of Deligne-Illusie varieties is closed under products to
this extent: if $X,Y$ are Deligne-Illusie varieties and if
$p>\dim(X)+\dim(Y)$ then $X\times_kY$ is a Deligne-Illusie
variety. 
\end{remark}

\subsection{Hodge de Rham degeneration and the Cartier operator} Recall from \cite[Section 1.1, page 7]{ogus79}
that thanks to the Cartier operator the Hodge de Rham spectral
sequence
\begin{equation}
E_1=H^q(X,\Omega^p_{X/k})\implies H^{p+q}_{\dR}(X/k)
\end{equation}
 degenerates at
$E_1$ if and only if the conjugate spectral sequence
\begin{equation}
E_2^{p,q}=H^p(X,{\mathbb H}_{\dR}^q(\Omega^\sbull_{X/k}))\implies
H^{p+q}_{\dR}(X/k)
\end{equation}
degenerates at $E_2$. The first spectral sequence, in any case,
induces the Hodge filtration on the abutment while the second
induces the conjugate Hodge filtration (see \cite{katz72}). In
particular we see that if $X$ is Mazur-Ogus then both the  Hodge
de Rham and the conjugate spectral sequences degenerate at $E_1$
(and $E_2$ resp.). In any case $H^*_{dR}(X/k)$ comes equipped with
two filtrations: the Hodge and the conjugate Hodge filtration.

\subsection{Locally closed and locally exact forms} Let $X$ be a smooth, projective surface over
an algebraically field $k$ of characteristic $p>0$. Then we have
the exact sequence
$$0\to B_1\Omega^1_X\to Z_1\Omega^1_X \to \Omega^1_X\to 0,$$
where the arrow $Z_1\Omega^1_X\to \Omega^1_X$ is the inverse
Cartier operator. In particular we have the subspace
$H^0(X,Z_1\Omega^1_X)\subset H^0(X,\Omega^1_X)$ which consists of
closed global one forms on $X$. Using iterated Cartier operators
(or their inverses), we get (see \cite[Chapter 0, 2.2, page
519]{illusie79b}) a sequence of sheaves $B_n\Omega^1_X\subset
Z_n\Omega^1_X$ and the exact sequence
$$0 \to B_n\Omega^1_X\to Z_n\Omega^1_X\to \Omega^1_X\to 0,$$
and sequence of sheaves $Z_{n+1}\Omega^1_X\subset Z_n\Omega^1_X$.
We will write
$$Z_\infty\Omega^1_X=\cap_{n=0}^\infty
Z_n\Omega^1_X.$$ This is the sheaf of indefinitely closed one
forms. We will say that a global one form is indefinitely closed
if it lives in $H^0(X,Z_\infty\Omega^1_X)\subset
H^0(X,\Omega^1_X)$. In general the inclusions
$H^0(Z_\infty\Omega^1_X)\subset H^0(Z_1\Omega^1_X)\subset
H^0(\Omega^1_X)$ may all be strict.

\subsection{Mazur-Ogus explicated for surfaces} For a smooth, projective surface, the condition that
$X$ is Mazur-Ogus takes more tangible geometric  forms which are
often easier to check in practice. Part of our next result is
implicit in \cite{illusie79b}. We will use the class of Mazur-Ogus
surfaces in extensively in this paper as well as its sequel and in
particular the following result will be frequently used.
\begin{theorem}
Let $X$ be a smooth, projective surface over a
perfect field $k$ of characteristic $p>0$. Consider the following
assertions
\begin{enumerate}
\item $X$ is Mazur-Ogus
\item $H^2_{cris}(X/W)$ is torsion free and $h^{1,1}=h^{1,1}_W$,
\item $H^2_{cris}(X/W)$ is torsion free and every global $1$-form
on $X$ is closed,
\item $\Pic(X)$ is reduced and every global $1$-form on $X$ is
indefinitely closed,
\item the differentials $H^1(X,\O_X)\to H^1(X,\Omega^1_X)$ and
$H^0(X,\Omega^1_X)\to H^0(X,\Omega^2_X)$ are zero,
\item the Hodge de Rham spectral sequence
$$E_1^{p,q}=H^{q}(X,\Omega^p_{X/k})\Rightarrow H^{p+q}_{dR}(X/k)$$
degenerates at $E_1$,
\end{enumerate}
Then (1) $\Leftrightarrow$ (2) $\Leftrightarrow$ (3) $\Rightarrow$
(4) $\Rightarrow$ (5) $\Leftrightarrow$ (6).
\end{theorem}
\begin{proof}
It is clear from  Ekedahl's Theorems (see \ref{ekedahl-theorems})
that (1) $\Rightarrow$ (2).  
Consider the assertion (2) $\Rightarrow$
(3). By \cite[Proposition 5.16, Page 632]{illusie79b} the
assumption of (2) that $H^2_{cris}(X/W)$ is torsion-free implies
that $\Pic(X)$ is reduced, the equality $$\dim
H^0(Z_1\Omega^1_X)=\dim H^0(Z_\infty\Omega^1_X),$$ and also the
equality
$$b_1=h^1_{dR} = h^{0,1}+\dim H^0(Z_1\Omega^1_X).$$
Now our assertion will be proved using Crew's
formula~\ref{crew-formula}. We claim that the following equalities
hold
\begin{eqnarray}
h^{0,0}_W&=&h^{0,0}\\
h^{0,1}_W&=&h^{0,1}\\
h^{0,2}_W&=&h^{0,2}\\
h^{2,0}_W&=&h^{2,0}\\
h^{1,1}-h^{1,1}_W&=&2(h^{1,0}-h^{0,1})
\end{eqnarray}
The first of these is trivial as both the sides are equal to one
for trivial reasons. The second follows from the explicit form of
Crew's formula for surfaces (and the fact that $\Pic(X)$ is
reduced). The third formula follows from Crew's formula and the
first two computations as follows. Crew's formula
\ref{crew-formula} says that
$$h^{0,0}_W-h^{0,1}_W+h^{0,2}_W=h^{0,0}-h^{0,1}+h^{0,2}.$$
By the first two equalities we deduce that $h^{0,2}_W=h^{0,2}$. By
Hodge-Witt symmetry and Serre duality we deduce that
$$h^{0,2}_W=h^{2,0}_W=h^{0,2}=h^{2,0}.$$
Again Crew's formula also gives
$$h^{1,0}_W-h^{1,1}_W+h^{1,2}_W=h^{1,0}-h^{1,1}+h^{1,2}$$
So we get on rearranging that
$$h^{1,1}-h^{1,1}_W=h^{1,0}-h^{1,0}_W+h^{1,2}-h^{1,2}_W.$$
By Serre duality $h^{1,2}=h^{1,0}$ and on the other hand
$$h^{1,2}_W=h^{1,0}_W=h^{0,1}_W=h^{0,1}$$
by  Hodge-Witt symmetry (see \ref{hodge-witt-symmetry},
\ref{hodge-witt-duality}) and the last equality holds as $\Pic(X)$
is reduced. Thus we see that
$$h^{1,1}-h^{1,1}_W=2(h^{1,0}-h^{0,1}).$$
Thus the hypothesis of (2) implies that
$$h^{1,0}=h^{0,1},$$ and as
$$h^{1,0}=\dim H^0(Z_1\Omega^1_X)=\dim
H^0(Z_\infty\Omega^1_X).$$ So we have deduced that every global
one form on $X$ is closed and hence (2) $\Rightarrow$ (3) is
proved.

Now (3) $\Rightarrow$ (1) is proved as follows. The only condition
we need check is that the hypothesis of (3) imply that Hodge de
Rham spectral sequence degenerates. The only non-trivial part of this assertion is
that $H^1(\O_X)\to H^1(X,\Omega^1_X)$ is zero. By
\cite[Prop.~5.16, Page 632]{illusie79b} we know that if
$H^2_{cris}(X/W)$ is torsion free then the differential
$H^1(\O_X)\to H^1(X,\Omega^1_X)$ is zero. Further by hypothesis of
(3) we see that $H^0(X,\Omega^1_X)\to H^0(X,\Omega^2_X)$ also is
zero. The other differentials in the Hodge de Rham spectral
sequence are either zero for trivial reasons or are dual to one of
the above two differentials and hence Hodge de Rham degenerates.
So (3) $\Rightarrow$  (1).

Now let us prove that (3) $\Rightarrow$ (4). The first assertion
is trivial after \cite[Prop.~5.16, Page 632]{illusie79b}.
Indeed the fact that $H^2_{cris}(X/W)$ is torsion free implies
that $\Pic(X)$ is reduced and
$H^0(X,Z_\infty\Omega^1_X)=H^0(X,Z_1\Omega^1_X)$ and by the
hypothesis of (3) we have further that
$H^0(X,Z_1\Omega^1_X)=H^0(X,\Omega^1_X)$. Thus we have deduced (3)
$\Rightarrow$(4).

The remaining assertions are well-known and are implicit in
\cite[Prop.~5.16, Page 632]{illusie79b} but we give a proof for
completeness. Now assume (4) we want to prove (5). By the
hypothesis of (4) and \cite[Prop.~5.16, Page 632]{illusie79b} we
see that the differential $H^0(X,\Omega^1_X)\to H^0(X,\Omega^2_X)$
is zero. So we have to prove that the differential $H^1(\O_X)\to
H^1(\Omega^1_X)$ is zero. We use the method of proof of
\cite[Prop.~5.16, Page 632]{illusie79b} to do this. Let $f:X\to
\Alb(X)$ be the Albanese morphism of $X$. Then we have a
commutative diagram
\begin{equation}
\xymatrix{
  0  \ar[r]^{} & H^0(A,\Omega^1_A) \ar[d]_{} \ar[r]^{} & H^1_{dR}(A/k)
  \ar[d]_{}
  \ar[r]^{} & H^1(\O_A) \ar[d]_{} \ar[r]^{} & 0 \\
  0 \ar[r]^{} & H^0(Z_1\Omega^1_X) \ar[r]^{} & H^1_{dR}(X/k) \ar[r]^{} & H^1(\O_X) &
  }
\end{equation}
with exact rows and  the vertical arrows are injective and
cokernel of the middle arrow is ${}_pH^2_{cris}(X/W)_{Tor}$ (the
$p$-torsion of the torsion of $H^2_{cris}(X/W)$). Moreover the
image of $H^1_{dR}(X/k)\to H^1(\O_X)$ is the $E_\infty^{0,1}$ term
in the Hodge de Rham spectral sequence.  Thus the hypothesis of
(4) that $\Pic(X)$ is reduced  implies   that
$H^1(\O_A)=H^1(\O_X)$, so we have $H^1(\O_A)=H^1(\O_X)\subset
E^{0,1}_\infty=H^1(\O_X)$. Hence $E^{0,1}_\infty=H^1(X,\O_X)$, so
that the differential $H^1(\O_X)\to H^1(\Omega^1_X)$ is zero. This
proves (4) implies (5).

Now let us prove (5) $\Leftrightarrow$ (6). It is trivial that (6)
implies (5). So we only have to prove  (5) $\Rightarrow$ (6). This
is elementary, but we give a proof. The differential $H^0(\O_X)\to
H^0(\Omega^1_X)$ is trivially zero, so by duality
$H^2(\Omega^1_X)\to H^2(\Omega^2_X)$ is zero. The differential
$H^0(\Omega^1_X)\to H^0(\Omega^2_X)$ is zero by hypothesis of (5).
This is dual to $H^2(\O_X)\to H^2(\Omega^1_X)$ hence which is also
zero. The differential $H^1(\Omega^1_X)\to H^1(\Omega^2_X)$ is
dual to the differential $H^1(\O_X)\to H^1(\Omega^1_X)$ which is
zero by hypothesis of (5). Hence we have proved (5)
$\Leftrightarrow$ (6).
\end{proof}

\subsection{Domino numbers of Mazur-Ogus varieties}\label{domino-numbers}
Let $X$ be a smooth projective variety over a perfect field. The
purpose here is to prove the following.  This was proved for
abelian varieties in \cite{ekedahl-diagonal}.

\begin{theorem}\label{domino-numbers-for-mazur-ogus}
        Let $X$ be a smooth projective Mazur-Ogus variety over a perfect
field. Then for all $i,j\geq 0$ the domino numbers $T^{i,j}$ are
completely determined by the Hodge numbers of $X$ and the slope
numbers of $X$.
\end{theorem}

\begin{proof}
This proved by an inductive argument. The first step is to note
that by the hypothesis and \cite{ekedahl-diagonal} one has
\begin{equation}
h^{i,j}=h^{i,j}_W=m^{i,j}+T^{i,j}-2T^{i-1,j+1}+T^{i-2,j+2}
\end{equation}
and so we get for all $j\geq 0$
\begin{equation}
T^{0,j}=h^{0,j}-m^{0,j}
\end{equation}
so the assertion is true for $T^{0,j}$ for all $j\geq 0$. Next we
prove the assertion for $T^{i,n}$ for all $i$. From the above
equation we see that
$T^{i,n}=h^{i,n}-m^{i,n}+2T^{i-1,n+1}-T^{i-2,n+2}$ and the terms
involving $n+1,n+2$ are zero. Now do a downward induction on $j$
to prove the result for $T^{i,j}$: for each fixed $j$, the formula
for $T^{i,j}$ involves $T^{i-1,j+1}$, $T^{i-2,j+2}$ and by
induction hypothesis on $j$ (for each $i$) these two domino
numbers are completely determined by the Hodge and slope numbers.
Thus the result follows.
\end{proof}
In particular as complete intersections in projective space are
Mazur-Ogus, we have the following.
\begin{corollary}
Let $X$ be a smooth projective complete intersection in projective
space. Then $T^{i,j}$ are completely determined by the Hodge
numbers of $X$ and the slope numbers of $X$.
\end{corollary}

\subsection{Hodge and Newton polygons of Hodge-Witt surfaces}
Let $X/k$ be a smooth, projective surface over a perfect field of characteristic $p>0$ with $H^*_{cris}(X/W)$  torsion-free. Let $\beta_2=\dim H^2_{dR}(X/k)=\dim_K H^2_{cris}(X/W)\tensor_W K$ be the second Betti number of $X$. Recall from \cite{katz79} that the Hodge polygon  of $H^2_{dR}(X/k)$ (resp. the Newton polygon of $H^2_{cris}(X/W)$) is the graph of a continuous, piecewise linear, $\R$-valued function denoted here by $\hod(X):[0,\beta_2]\to\R$ (resp. $\new(X):[0,\beta_2]\to\R$). We refer the reader to \cite{berthelot-ogus,katz79} for the definition of these polygons and functions. 

Readers may find the following result which is, as far as we are aware, the most explicit result of its kind, useful in understanding  the influence of Hodge-Witt condition  in the more familiar realm of Hodge and Newton polygons.
\bthm\label{th:hodge-newton-polygons} 
Let $X/k$ be a smooth, projective surface over a perfect field $k$ of characteristic $p>0$. Suppose $X$ is Hodge-Witt and $H^*_{cris}(X/W)$ is torsion-free. Then 
$$\hod(x)=\new(x) \text{ on the interval } \sum_{\lambda<1}m_\lambda \leq x\leq \sum_{\lambda<1}m_\lambda + m_1=\sum_{\lambda\leq1}m_\lambda.$$
In other words the Hodge and Newton polygons of $X$ touch over the closed interval $$\sum_{\lambda<1}m_\lambda \leq x\leq\sum_{\lambda<1}m_\lambda + m_1=\beta_2-\sum_{\lambda<1}m_\lambda.$$
\ethm

\begin{remark}
Here is an explicit example which illustrates this theorem. Let $X$ be Hodge-Witt  a K3 surface (equivalently a K3 surface of finite height). Then $H^2_{cris}$ has only one slope $\lambda=1-\frac{1}{n}<1$ of multiplicity $m_\lambda=n$ where $1\leq n\leq 10$, and one sees easily that $m_1=22-2n$. Hence the theorem asserts that Hodge and Newton functions of $X$ agree on the interval $[n,n+22-2n]=[n,22-n]$ an interval of length at least two as $n\geq 1$.
\end{remark}

\bp 
Note that projectivity hypothesis implies $m_1\geq 1$ and hence the interval in the assertion has length at least one. By the definition of the Hodge-function one has $\hod(x)=0$ on $[0,h^{0,2}]$ and $\hod(x)$ is linear of slope one on $[h^{0,2},h^{0,2}+h^{1,1}]$ and of slope two on the interval $[h^{0,2}+h^{1,1},h^{0,2}+h^{1,1}+h^{2,0}]=[h^{0,2}+h^{1,1},\beta_2]$. From this one has
$$\hod(x)=\begin{cases}
0 & \text{ if } 0\leq x\leq h^{0,2},\\
x-h^{0,2} & \text{ if } h^{0,2}\leq x\leq h^{0,2}+h^{1,1},\\
2x-\beta_2 & \text{ if }  h^{0,2}+h^{1,1}\leq x\leq h^{0,2}+h^{1,1}+h^{2,0}=\beta_2.\\
\end{cases}
$$
Now let us determine the Newton function of $H^2_{cris}(X/W)$. Let $(\lambda_1,n_1),\ldots,(\lambda_r,n_r)$ be all the distinct slopes of $(H^2_{cris}(X/W),F)$. So the slope sequence looks like this:
$$\overbrace{\lambda_1,\cdots,\lambda_1}^{n_1\text{ times }},\overbrace{\lambda_2,\cdots,\lambda_2}^{n_2\text{ times }},\ldots,\overbrace{\lambda_r,\cdots,\lambda_r}^{n_r\text{ times }},$$
and one of these slopes is equal to $1$ and its multiplicity we will denote by $m_1$. Now by definition of the Newton polygon (or the Hodge function) one has
$$\new(x)=\begin{cases}
\lambda_1x & \text{ for } 0\leq x\leq n_1,\\
\lambda_2(x-n_1)+\lambda_1n_1 & \text{ for } n_1\leq x\leq n_1+n_2,\\
\qquad\vdots & \qquad \vdots \\
\lambda_{i+1}\left(x-\sum_{j=1}^i n_j\right)+\sum_{j=1}^i\lambda_jn_j & \text{ for } \sum_{j=1}^in_j\leq x\leq \sum_{j=1}^{i+1}n_j,\\
\qquad\vdots & \qquad\vdots \\
\end{cases}
$$
Now suppose $\lambda_1<\lambda_2<\cdots \lambda_{i_0}<\lambda_{i_0+1}=1<\lambda_{i_0+2}<\cdots<\lambda_r$. Then on the closed interval
$$\sum_{j\leq i_0}n_j\leq x\leq\sum_{j\leq i_0}n_j+n_{i_0+1}=\sum_{j\leq i_0}n_j+m_1$$ one has
\beas
\new(x)&=&1\cdot \left(x-\sum_{j=1}^{i_0} n_j\right)+\sum_{j=1}^{i_0}\lambda_jn_j\\
&=&x-\sum_{j=1}^{i_0}(1-\lambda_j)n_j\\
&=& x-m^{0,2}.
\eeas
Now $X$ is Hodge-Witt one has $T^{0,2}=0$ and as $H^2_{cris}(X/W)$ is torsion-free one has from \cite[Corollary 5]{crew85} that
$$m^{0,2}=h^{0,2},$$
and hence $$\new(x)=x-m^{0,2}=x-h^{0,2}=\hod(x) $$ 
for $\sum_{\lambda<1}m_\lambda \leq x\leq\sum_{\lambda<1}m_\lambda + m_1$.
Further by Poincar\'e duality and our hypothesis, one has $\beta_2=2\sum_{\lambda<1}m_\lambda+m_1$. Hence
$$\beta_2-\sum_{\lambda<1}m_\lambda=\sum_{\lambda<1}m_\lambda+m_1.$$ This proves the assertion.
\ep

\subsection{Explicit formulae for domino numbers of hypersurfaces}\label{explicitformulas}
For $X\subset \P^{n+1}$ a smooth hypersurface we can make the
formulas quite explicit using \cite{lewis-survey,rapoport72,suwa93}.
We do this here for the reader's convenience.
We begin with the remark that we have the following explicit bound
for $T^{0,j}$:
\begin{equation}
T^{0,j}\leq h^{0,j},
\end{equation}
and if $\dim(X)=n=2$ then we have $m^{0,2}=0$ if and only if
$H^2_{cris}(X/W)\tensor_WK$ is pure slope one. In particular, if this is the case, we have
$T^{0,2}=h^{0,2}$ and so we have $T^{0,2}=p_g(X)$. If $\deg(X)=d$
then we have from \cite{lewis-survey} that $h^{0,2}=\dim
H^0(X,\Omega^2_X)=H^0(X,\O_X(d-4))$ because $\Omega^2_X=\O_X(d-4)$
by the adjunction formula.

If $\dim(X)=3$ one sees from \cite{suwa93} that there are at most two, possibly nonzero domino numbers, 
namely, $T^{0,2}$ and $T^{0,3}$; the remaining domino numbers are either zero or
equal to these two by Ekedahl's duality. This is seen as follows:  as $X$ is a hypersurface
$H^{i}(X,\O_X)=0$ for $0<i<\dim(X)$ and hence $T^{0,i}=0$ for all
$i$ except possibly $i=3$ and $T^{1,2}=T^{0,3}$ by Ekedahl's
duality. So we have $T^{0,3}=h^{0,3}-m^{0,3}$. If
$H^3_{cris}(X/W)$ has no slopes in $[0,1)$ then we see that
$m^{0,3}=0$ and so again $T^{0,3}=h^{0,3}=\dim H^3(X,\O_X)$.

If $\dim(X)=4$ then from \cite{suwa93} one sees that there are two possibly non-trivial domino numbers to
determine $T^{0,4}$ and $T^{1,3}$ and we may determine $T^{0,4}$ easily as above.
The formula for $h^{1,3}_W=m^{1,3}+T^{1,3}-2T^{0,4}$ gives
$T^{1,3}=h^{1,3}-m^{1,3}+2T^{0,4}$. If $H^4_{cris}(X/W)$ is pure
slope two, then $m^{1,3}=0$ and so we have
$T^{1,3}=h^{1,3}+2h^{0,4}.$

\begin{description}
  \item[$\dim(X)=2$] We have $T^{0,2}\leq h^{0,2}$ and equality
  holds  if and only if $H^2_{cris}(X/W)$ is pure slope one.
  \item[$\dim(X)=3$] We have $T^{0,3}\leq h^{0,3}$ with equality if
  and only if $H^3_{cris}(X/W)$ has no slopes in $[0,1)$ (equivalently by Poincar\'e duality
  $H^3_{cris}(X/W)$ has no slopes in $(2,3]$).
  \item[$\dim(X)=4$] We have $T^{0,4}\leq h^{0,4}$ and $T^{1,3}\leq
  h^{1,3}+2T^{0,4}$ and both are equal if and only if
  $H^4_{cris}(X/W)$ is of pure slope two.
\end{description}
Thus we have proved the following:

\begin{corollary}
Let $X\subset\P^{n+1}$ be a smooth, projective hypersurface over a
perfect field of characteristic $p>0$ with $2\leq\dim(X)=n\leq 4$.
Then we have $T^{0,i}=0$ unless $i=n$ and in that case
\begin{equation}\label{T0ibound}
	T^{0,n}\leq h^{0,n}.
\end{equation}
Further $T^{1,i}=0$ unless $i=4$ and if $i=4$ then
\be 
	T^{1,4}\leq h^{1,3}+2h^{0,4}.
\ee
Further equality holds if $H^i_{cris}(X/W)$ satisfies certain slope
conditions which are summarized in Table~\vref{tab:slopes-table}. The table
records the only, possibly non-trivial, domino numbers and the
crystalline condition which is necessary and sufficient for these
domino numbers to achieve their maximal value.
\begin{center}
\begin{table}[H]
\caption{Slope condition(s) for maximal domino numbers\label{tab:slopes-table}}
{\renewcommand{\arraystretch}{1.5}
\begin{tabular}{|c|c|c|c|}
  \hline
  $\dim(X)$ & $2$ & $3$ & $4$ \\
  \hline
  $T^{0,2}$ & $H^2_{cris}=H^2_{cris,[1]}$ & $0$ & $0$ \\
  \hline
  $T^{0,3}$ & $0$ & $H^3_{cris}=H^3_{cris,[1,2]}$ & $0$ \\
  \hline
  $T^{0,4}$ & $0$ & $0$ & $H^4_{cris}=H^4_{cris,[1,3]}$ \\
  \hline
  $T^{1,3}$ & $0$ & $0$ & $H^4_{cris}=H^4_{cris,[2]}$ \\
  \hline
\end{tabular}
} 
\end{table}
\end{center}
\end{corollary}

The following table records the standard formulae (see
\cite{deligne-complete,lewis-survey,rapoport72}) for the non-trivial numerical
invariants of a smooth hypersurface $X\subset \P^{n+1}$  with $2\leq
n\leq 4$ and of degree $d$. These can extracted as  the coefficients
of the power series expansion of the function
\begin{equation}
H_d(y,z)=\frac{(1+z)^{d-1}-(1+y)^{d-1}}{z(1+y)^d-y(1+z)^d}=\sum_{p,q\geq
0}h_0^{p,q}y^pz^q,
\end{equation}
where $h^{p,q}_0=h^{p,q}-\delta_{p,q}$ (where $\delta_{p,q}$ is the Kronecker delta symbol); we may also calculate the
Betti number of $X$ using the Hodge numbers of $X$ as Hodge de Rham spectral sequence degenerates for a hypersurface in $\P^{n+1}$.

For computational purposes, we can write the above generating function, following \cite{deligne-complete}, the sum as
\be 
H_d(y,z)=\frac{\sum_{i,j\geq0}\binom{d-1}{i+j+1}y^iz^j}{1-\sum_{i,j\geq1}\binom{d}{i+j}y^iz^j}
\ee
Now it is possible, after a bit of work (which we suppress here) to arrive at the formulae for Hodge numbers for low dimensions (such as the ones we need). Our results are summarized in the following table.

Assuming that the slope conditions for maximal domino numbers hold (see Table~\vref{tab:slopes-table}) we can use the Hodge number calculation to calculate domino numbers. Table~\vref{hodge-table} gives formulae for Hodge numbers $h^{i,j}$.
\begin{center}
\begin{table}[H]
\caption{Hodge and Betti Numbers\label{hodge-table}}
{\renewcommand{\arraystretch}{1.5}
\begin{tabular}{|c|c|c|c|}
  \hline
  $\dim(X)$ & $2$ & $3$ & $4$ \\
  \hline
  $h^{0,2}$ & $\frac{(d-1)(d-2)(d-3)}{6}$ &  &  \\
  \hline
   $h^{1,1}$ & $\frac{d(2d^2 - 6d + 7)}{3}$ &  &  \\
  \hline
  $h^{0,3}$ &  & $\frac{(d-1)(d-2)(d-3)(d-4)}{4!}$ &  \\
  \hline
  $h^{1,2}$ &  & $\frac{(d-1)(d-2)(11d^2-17d+12)}{4!}$ &  \\
  \hline
  $h^{0,4}$ &  &  & $\frac{(d-1)(d-2)(d-3)(d-4)(d-5)}{5!}$ \\
  \hline
  $h^{1,3}$ &  &  & $\frac{2(d - 1)(d - 2)(13d^3 - 51d^2 + 56d - 30)}{5!}$ \\
  \hline
  $h^{2,2}$ &  &  & $\frac{(d - 1)(d - 2)(3d^3 - 11d^2 + 11d - 5)
}{10}$ \\
  \hline
  $b_2$ & $d^3 - 4d^2 + 6d - 2$
  &  &  \\
  \hline
  $b_3$ &  & $(d - 1)(d - 2)(d^2 - 2d + 2)$  &  \\
  \hline
  $b_4$ & & & $\frac{(d - 1)(d - 2)(3d^3 - 12d^2 + 15d - 10)}{4}$ \\
  \hline
\end{tabular}
} 
\end{table}
\end{center}
We record for future use:
\be\label{h11and2pg}
h^{1,1}-2p_g=\frac{d^3 - 4d + 6}{3}
\ee
and hence
\be\label{b2and4pg} 
b_2-4p_g=\frac{d^3 - 4d + 6}{3}.
\ee

We summarize formulae for maximal values of $T^{0,i}$ which we can obtain using this method in the following.

\begin{proposition}
Assume that $X\subseteq \P^n$ is a smooth, projective hypersurface of degree $d$ and $\dim(X)\leq 4$. Suppose that the crystalline cohomology of $X$ satisfies the slope condition of for maximal domino numbers given in  (Table~\vref{tab:slopes-table}). Then the domino numbers $T^{0,i}$ for $i=2,3,4$ (resp. $\dim(X)=2,3,4$) are given by Table~\vref{domino-table}:
\begin{center}
\begin{table}[H]
\caption{Maximal domino numbers\label{domino-table}}
{\renewcommand{\arraystretch}{1.5}
\begin{tabular}{|c|c|}
\hline
$i$ & $T^{0,i}$ \\
\hline
2 &  $\frac{(d-1)(d-2)(d-3)}{6}$ \\
\hline
3 &  $\frac{(d-1)(d-2)(d-3)(d-4)}{4!}$ \\
\hline
4 &  $\frac{(d-1)(d-2)(d-3)(d-4)(d-5)}{5!}$ \\
\hline
$T^{1,3}$ & $\frac{(d - 1)(d - 2)(14d^3 - 63d^2 + 103d - 90)}{5!}$ \\
\hline
\end{tabular}}
\end{table}
\end{center}
\end{proposition}

\section{Enriques Classification and negativity of $h^{1,1}_W$}
\label{hodge-witt-of-surfaces}
\subsection{Main result of this section}
The main theorem we want to prove is Theorem~\ref{negative-hw11}.
The proof of Theorem~\ref{negative-hw11} is divided in to several
parts and it uses the Enriques classification of surfaces. We do
not know how to prove the assertion without using Enriques
classification \cite{bombieri76}, \cite{bombieri77}.

\begin{theorem}\label{negative-hw11}
Let $X/k$ be a smooth, projective surface over a perfect field $k$
of characteristic $p>0$. Suppose $h^{1,1}_W<0$ then the following dichotomy holds:
\benum
\item $X$ is quasielliptic, or
\item $X$   is of general type and further
\benum
\item either $c_2\leq 0$, or  
\item $\Omega^1_X$ is Bogomolov unstable
\eenum
\eenum  
%
%
\end{theorem}

\begin{remark}
Recall that the quasielliptic surfaces exist if and only if $p=2$ or $p=3$ and such surfaces are of  Kodaira dimension one. In particular if $p\geq 5$ and $h^{1,1}_W<0$ then $X$ is of general type.	For $p=2,3$  using \cite[Corollary, page 480]{lang79} and the formula for $h^{1,1}_W$ it
is possible to write down examples of quasi-elliptic surfaces of
Kodaira dimension one where this invariant is negative.
\end{remark}

\subsection{Reduction to minimal model} Before proceeding further we record the following elementary lemma which
allows us to reduce the question of $h_W^{1,1}<0$ to minimal
surfaces (when such a model exists). 

\begin{lemma}\label{minimal-reduction}
	If $X$ is a smooth, projective surface over a perfect field with
	$h^{1,1}_W(X)<0$, and if $X\to X'$ is a proper birational morphism with $X'$ minimal, then
	$h^{1,1}_W(X')<0$.
\end{lemma}
\begin{proof}
	Since every proper birational morphism $X\to X'$ is a composition of a finite number of blowups at closed points, it is enough to prove the assertion for a blowup at a  closed point. Now the lemma follows from the easily established fact that under blowup at a closed point,  $h^{1,1}_W$ increases by the degree of the point (this fact can be easily established by using standard properties of Chern classes and Betti nunbers under blowups as applied to \eqref{ekedahl-formula}), so
	passing from $X$ to $X'$ involves a decrease in $h_W^{1,1}$ by this degree. Thus we
	see that  $h^{1,1}_W(X')<h^{1,1}_W(X)<0$.
\end{proof}

\subsection{Enriques classification}\label{enriques}
        We briefly recall Enriques classification of surfaces
(\cite{mumford69}, \cite{bombieri77}, \cite{bombieri76}). 
Let $X/k$ be smooth, projective surface over $k$. Then Enriques
classification is carried out by means of the Kodaira dimension
$\kappa(X)$. All surfaces with $\kappa(X)=-\infty$ are ruled
surfaces; the surfaces with $\kappa(X)=0$ comprise of K3 surfaces,
abelian surfaces, Enriques surfaces, non-classical Enriques
surfaces (in characteristic two), bielliptic surfaces and
non-classical hyperelliptic surfaces (in characteristic two and
three). The surfaces with $\kappa(X)=1$ are (properly) elliptic
surfaces and finally the surfaces with $\kappa(X)=2$ are  surfaces
of general type.

\subsection{Two proofs of non-negativity}  In this
subsection we give two proofs of the following:
\begin{proposition}\label{non-negative-hw-for-surfaces}
Let $X/k$ be a  geometrically connected, smooth projective surface.
Then for $(i,j)\neq (1,1)$ we have $h^{i,j}_W\geq 0$.
\end{proposition}

\begin{proof}[First proof]
The assertion is trivial for $h^{0,1}_W=h^{1,0}_W=b_1/2$. So we have
to check it for $h^{2,0}_W=h^{0,2}_W=\chi(\O_X)-1+b_1/2$. Writing out
this explicitly we have
\begin{equation}
h^{0,2}_W=h^{0,0}-h^{0,1}+h^{0,2}-1+b_1/2
\end{equation}
or as $h^{0,0}=1$ (as $X$ is geom. connected) we get
\begin{equation}
h^{0,2}_W=h^{0,2}-(h^{0,1}-q)
\end{equation}
where $q=b_1/2=\dim_k\Alb(X)$ is the dimension of the Albanese variety
of $X$. By \cite[page 25]{bombieri76} we know that $h^{0,1}-q\leq
p_g=h^{0,2}$ and so the non-negativity assertion follows.
\end{proof}

\begin{proof}[Second Proof]
We use
the definition of
\be 
	h^{i,j}_W=m^{i,j}+T^{i,j}-2T^{i-1,j+1}+T^{i-2,j+2}.
\ee To prove the
result it suffices to show that $T^{i-1,j+1}$ is zero for all
$(i,j)\neq (1,1)$. This follows from the fact that in the slope
spectral sequence of a smooth projective surface, there is at most one
non-trivial differential (see \cite{nygaard79b}, \cite[Corollary 3.14,
page 619]{illusie79b}) and this gives vanishing of the domino numbers
except possibly $T^{0,2}$, and if $(i-1,j+1)=(0,2)$ then
$(i,j)=(1,1)$.
\end{proof}

\subsection{The case $\kappa(X)=-\infty$} We begin by
stepping through the Enriques classification (see \ref{enriques}) and
verifying the non-negativity of the Hodge-Witt number in all the
cases.
\begin{proposition}\label{kappaminusinfinity}
Let $X/k$ be a smooth, projective surface. If $\kappa(X)=-\infty$
then $h^{1,1}_W\geq 0$.
\end{proposition}

\begin{proof}
Since every smooth, projective surface admits a birational morphism to a  smooth,  minimal surface, and every such morphism is composed of a finite number of blowups at closed points (which increase $h^{1,1}_W$), to prove non-negativity of $h^{1,1}_W$ we may assume that $X$ is minimal with $\kappa(X)=-\infty$.	
As $\kappa(X)=-\infty$, we know from \cite{bombieri77} that either $X$
is rational or it is ruled (irrational ruled). Assume $X$ is
irrational ruled. Then one has $c_1^2=8-8q$ and $\chi(\O_X)=1-q$. By
Noether's formula $\chi(\O_X)=\frac{1}{12}(c_1^2+c_2)$ we get
\begin{equation}
h^{1,1}_W=b_1+\frac{5}{6}4(1-q)-\frac{1}{6}8(1-q)=b_1+2(1-q)=2\geq0
\end{equation}
where we have used the fact that for a ruled surface $\Pic(X)$ is
reduced (which follows from $p_g=0$ for a ruled surface so
$H^2(X,W(\O_X))=0$), and the fact that $b_1=2q$.

If $X$ is rational, then either $X=\P^2$ or $X$ is ruled, rational. In
the first case $c_1^2=9$ and in the second case $c_1^2=8$. In both the
cases $\chi(\O_X)=1$ and we are done by an explicit calculation.
\end{proof}

\subsection{The case $\kappa(X)=0$}
\begin{proposition}\label{kappazero}
Let $X/k$ be a smooth projective, minimal surface with
$\kappa(X)=0$. Then $h^{1,1}_W\geq 0$.
\end{proposition}
\begin{proof}
This is easy: when $\kappa(X)=0$, we know that $c_1^2=0$ and so it
suffices to show that $\chi(\O_X)\geq 0$. This follows from the table
in \cite[page 25]{bombieri76}.
\end{proof}

\subsection{The case $\kappa(X)=1$}
The following proposition shows that $h^{1,1}_W\geq 0$ holds for
surfaces of $\kappa(X)=1$ unless the surface is quasi-elliptic.
There are example of quasi-elliptic surfaces for which the result
fails.
\begin{proposition}\label{kappaone}
Assume $X$ is a smooth projective, minimal surface over a perfect
field of characteristic $p>0$ with $\kappa(X)=1$. If $p=2,3$,
assume that $X$ is not quasi-elliptic. Then $h^{1,1}_W\geq 0$.
\end{proposition}
\begin{proof}
Under our hypothesis,  $X$ is a properly elliptic surface (i.e.,
the generic fibre is smooth curve of genus $1$  and $c_1^2=0$.
Hence it suffices to verify that $c_2\geq 0$. As
$c_2=\chi_{et}(X)$, the required inequality is equivalent to
proving $\chi_{et}(X)\geq 0$. This inequality is implicit  in
\cite{bombieri76}; it can also be proved directly using the Euler
characteristic formula (see \cite{cossec-dolgachev}[page 290,
Proposition~5.1.6] and the paragraph preceding it).
\end{proof}
\subsection{Proof of Theorem~\ref{negative-hw11}}
Now we can assemble various components of the proof.  Assume that
$X$ has $h^{1,1}_W<0$. By Proposition~\ref{kappaminusinfinity} we
have $h^{1,1}_W\geq 0$ for $\kappa(X)=-\infty$, so we may assume
that $\kappa(X)\geq 0$. Then by Lemma~\ref{minimal-reduction} we
may assume that $X$ is already minimal. Now by
Proposition~\ref{kappazero} and Proposition~\ref{kappaone} we have
$h^{1,1}_W\geq 0$ if $0\leq \kappa(X)\leq 1$ and $X$ is not quasi-elliptic. So if $h^{1,1}_W<0$ and $X$ is not quasi-elliptic then one has $\kappa(X)>1$ and so $X$ is of general type.

So suppose $X$ is of general type with $h^{1,1}_W<0$. Observe that one has the tautology $c_2\leq 0$ or $c_2>0$. If the first of these holds then there is nothing to prove. So suppose that $X$ is of general type with $h^{1,1}_W<0$ and $c_2>0$ then we claim that $\Omega^1_X$ is Bogomolov unstable (see \cite{shepherd-barron91b} for the definition of Bogomolov stability). This follows from
	\cite[Corollary 15]{shepherd-barron91b}. To see that the
	conditions of that corollary are valid it suffices to verify that
	$c_1^2>\frac{16p^2}{(4p^2-1)}c_2$. By \eqref{ekedahl-formula}, one has $h^{1,1}_W<0\implies c_1^2>5c_2+6b_1>5c_2$. On the other hand by \cite[Corollary 15]{shepherd-barron91b} if $c_2>0$ and $c_1^2>\frac{16p^2}{4p^2-1}c_2$, then $\Omega_X^1$ is Bogomolov unstable. As
	\begin{equation}
	4< \frac{16x^2}{(4x^2-1)}\leq 4.26 <5\cdots
	\end{equation}
	for $x\geq 2$ our claim of Bogomolov unstability of $\Omega^1_X$ follows from Shepherd-Barron's
	result.

\begin{remark}
Surfaces of general type with $c_2\leq 0$ are studied in   \cite{shepherd-barron91b} where it is shown, for instance that if $c_2<0$ then $X$ is also uniruled (a result conjectured by Michel Raynaud).
\end{remark}

\section{Chern class inequalities}\label{chern-inequalities}
\subsection{Elementary observations}
In this section we study the Chern class inequality $c_1^2\leq
5c_2$ and a weaker variant $c_1^2\leq 5c_2+6b_1$. These were
studied in characteristic zero in \cite{deven76a}. It is, of course, well-known that
$c_1^2\leq 5c_2$ fails for some surfaces in positive characteristic. The first observation we
have, albeit an elementary one, is that the obstructions to
proving $c_1^2\leq 5c_2$ are of de Rham-Witt (i.e. involving
torsion in the slope spectral sequence) and crystalline (i.e.
involving slopes of Frobenius on $H^2_{cris}(X/W)$) in nature. This has not
been noticed before.

Let us begin by recording some trivial but important consequences
of the remarkable formula for $h^{1,1}_W$ (see
\ref{ekedahl-formula}). The main reason for writing them out
explicitly is to illustrate the fact that obstructions to Chern
class inequalities for surfaces are of crystalline (involving
slope of Frobenius) and de Rham-Witt (involving the domino number
$T^{0,2}$).

In what follows we will write
\beas
c_i&=&c_i(T_X)\\
b_1&=&\dim_K
H^1_{cris}(X/W)\tensor K,\\
T^{0,2}&=&\dim_k
\dom^{0,2}(H^2(X,W(\O_X))\to H^2(X,W\Omega^1_X)).
\eeas
\begin{proposition}\label{chern-inequalities2} Let $X$ be a smooth, projective surface over a
perfect field of characteristic $p>0$.
\begin{enumerate}
\item Then the following conditions are equivalent:
\begin{enumerate}
\item the inequality  $c^2_1\leq 5c_2$ holds,
\item the inequality $ h^{1,1}_W\geq b_1 $ holds (if $X$ is Mazur-Ogus this is equivalent to $h^{1,1}\geq b_1$),
\item the inequality $2T^{0,2}+b_1 \leq m^{1,1}$ holds.
\end{enumerate}
\item If $X$ is a Hodge-Witt surface. Then the
following assertions are equivalent
\begin{enumerate}
\item the inequality $c_1^2\leq 5c_2$ holds,
\item the inequality $m^{1,1}\geq 2m^{0,1}$ holds.
\end{enumerate}
\end{enumerate}
\end{proposition}
\begin{proof}
All the assertions are trivial consequences of the following formulae, and the fact that if $X$ is Hodge-Witt then $T^{0,2}=0$.
\begin{eqnarray}
h^{1,1}_W &=&m^{1,1}-2T^{0,2}\\
h^{1,1}_W &=&\frac{5c_2-c_1^2}{6}+b_1\\
b_1&=&m^{0,1}+m^{1,0}\\
m^{0,1}&=&m^{1,0},
\end{eqnarray}
and are left to the reader.
\end{proof}
\subsection{Consequences of $h^{1,1}_W\geq 0$}
We also record the main reason for our interest in $h^{1,1}_W\geq
0$.
\begin{proposition} Let $X/k$ be a smooth projective surface over a
perfect field $k$. Then
$$h^{1,1}_W\geq 0$$
holds if and only if the
inequality:
\begin{equation}\label{weakBMY}
c_1^2 \leq 5c_2+6b_1.
\end{equation}
holds. On the other hand if $h_W^{1,1}<0$, then
\begin{equation}\label{hodge-witt-fault-line}
c_1^2\geq 5c_2.
\end{equation}
\end{proposition}

\begin{proof}
The asserted inequality  follow easily from Ekedahl's formula
\eqref{ekedahl-formula} for $h^{1,1}_W$:
\begin{equation}
h^{1,1}_W=b_1+\frac{5}{6}c_2-\frac{1}{6}c_1^2
\end{equation}
Hence we see that $h^{1,1}_W\geq 0$ gives $h^{1,1}_W>0$ gives
$5c_2-c_1^2>-6b_1$ or $c_1^2< 5c_2+6b_1$.  Further we see that
$h^{1,1}_W<0$ implies that
\begin{equation}
b_1+\frac{5}{6}c_2-\frac{1}{6}c_1^2<0
\end{equation}
As $b_1\geq 0$ the term on the left is not less than
$\frac{5}{6}c_2-\frac{1}{6}c_1^2$ and so
\begin{equation}
\frac{5}{6}c_2-\frac{1}{6}c_1^2\leq h^{1,1}_W <0,
\end{equation}
and the result follows.
\end{proof}

\begin{remark} Let $X$ be a smooth projective surface of general
type. Clearly when $h^{1,1}_W<0$  the Bogomolov-Miyaoka-Yau
inequality also fails. On the other hand if $X$ satisfies $c_1^2\leq
3c_2$ then $h^{1,1}_W\geq 0$. Thus the point of view which seems
to emerge from the results of this section is that surfaces with
$h^{1,1}_W<0$ are somewhat more exotic than the ones for which
$h^{1,1}_W\geq 0$. Indeed as was pointed out in
\cite{ekedahl-diagonal}, $h^{1,1}_W$ is a deformation invariant so
surfaces with $h^{1,1}_W<0$ do not even admit deformations which
lift to characteristic zero.
\end{remark}

\begin{corollary} If $X$ is a smooth projective surface for which
\eqref{weakBMY} fails to hold, then the slope spectral sequence of
$X$ has infinite torsion and does not degenerate at $E_1$.
\end{corollary}
\begin{proof}
Indeed, this follows from the formula
\begin{equation}
h^{1,1}_W=m^{1,1}-2T^{0,2},
\end{equation}
which is just the definition of $h^{1,1}_W$. The claim now follows as
$m^{1,1}\geq 0$ and hence $h^{1,1}_W<0$ implies that $T^{0,2}\geq
1$.
\end{proof}

\begin{remark}
        Thus we see that the  counter examples to Bogomolov-Miyaoka-Yau inequality given
in \cite{szpiro79} are not Hodge-Witt.
\end{remark}

\subsection{The Surfaces for which $c_1^2\leq 5c_2+6b_1$ holds or equivalently $h^{1,1}_W\geq 0$ holds}
Our next result provides a large class of surfaces for which
$h^{1,1}_W\geq 0$ does hold.
\begin{theorem}\label{surfaces-with-non-negative-h11} Let $X/k$ be
a smooth, projective surface over a perfect field of
characteristic $p>0$. Assume $X$ satisfies any one of the
following hypothesis:
\begin{enumerate}
\item the surface  $X$ is Hodge-Witt (in the sense of (\ref{la:hodge-witt-ordinary-def})),
\item or $X$ is ordinary (in the sense of (\ref{la:hodge-witt-ordinary-def})),
\item or $X$ is a Mazur-Ogus surface,
\item or assume $p\geq 3$ and $X$ is a Deligne-Illusie surface,
\item or assume $p=2$, and $X$ lifts to $W_2$.
\end{enumerate}
Then $X$ satisfies \eqref{weakBMY}:
\bes 
c_1^2\leq 5c_2+6b_1
\ees
\end{theorem}
\begin{proof}
        The assertion that \eqref{weakBMY} holds is equivalent to
$h^{1,1}_W\geq 0$ where:
\bes
h^{1,1}_W=m^{1,1}-2T^{0,2}=\frac{5c_2-c_1^2}{6}+b_1.
\ees
 Thus it suffices to prove that $h^{1,1}_W\geq
0$ under the any of the assumptions (1)--(5). If (1) holds then the asserted inequality  follows from the fact that $T^{0,2}=0$ as $X$ is
Hodge-Witt and $m^{1,1}\geq 0$ by definition. If (2) holds then $X$ is ordinary and in particular $X$ is Hodge-Witt and so the result follows from the implication (1) $\implies$ \eqref{weakBMY}.   If (3) holds we can simply invoke \cite[Corollary 3.3.1, Page
86]{ekedahl-diagonal} which gives us  $h^{1,1}_W=h^{1,1}$. However
we give an elementary proof in the spirit of this paper. We will
use the formulas $\chi(\O_X)=1-h^{0,1}+h^{0,2}$ and
$c_2=\chi_{et}(X)=1-b_1+b_2-b_3+b_4=2-2b_1+b_2$. By
(\ref{noether-formula1}) we have
\begin{equation}
12\chi(\O_X)=c_1^2+c_2,
\end{equation}
or equivalently $c_1^2=12\chi(\O_X)-c_2$. Now the assertion will
follow if we prove that $c_1^2\leq 5c_2+6b_1$. But
\begin{eqnarray}
5c_2-c_1^2+6b_1& = &5c_2-(12\chi(\O_X)-c_2)+b_1\\
                &=&6c_2-12\chi(\O_X)+6b_1\\
                &=&6(2-2b_1+b_2)-12\chi(\O_X)+6b_1\\
                &=&12-12b_1+6b_2-12(1-h^{0,1}+h^{0,2})+6b_1\\
                &=&6b_1-12h^{0,1}+6b_2-12h^{0,2}
\end{eqnarray}
Thus we see that
$5c_2-c_1^2+6b_1=6(b_1-2h^{0,1})+6(b_2-2h^{0,2})$. By
\cite{deligne87} and the hypothesis that the crystalline
cohomology of $X$ is torsion free we have
$b_2=h^{0,2}+h^{1,1}+h^{2,0}$. Or equivalently by Serre duality we
get $b_2=2h^{0,2}+h^{1,1}$ and again by the hypothesis that the
crystalline cohomology of $X$ is torsion free we see that
$\Pic(X)$ is reduced and so $b_1=2h^{0,1}$. Thus
$5c_2-c_1^2+6b_1=6h^{1,1}$ and so is non-negative and in
particular we have deduced that $h^{1,1}_W=h^{1,1}$. 
Now the implication (4) $\implies$ \eqref{weakBMY} follows from
the implication  (3) $\implies$ \eqref{weakBMY} (via \cite{deligne87})--see Theorem~\ref{deligne-illusie-thm}.
The fifth 
assertion (5) $\implies$ \eqref{weakBMY} falls into two cases: assume $X$ is not ruled,
then  this follows from \cite{shepherd-barron91b} as the
hypothesis  (5) implies that $c_1^2\leq 3c_2$ by \cite{shepherd-barron91b}. If $X$ is ruled one
deduces this from our earlier result on surfaces with Kodaira
dimension $-\infty$.
\end{proof}

The following corollary is immediate:
\begin{corollary}\label{surfaces-with-non-negative-h11-cor}
Under the hypothesis of Theorem~\ref{surfaces-with-non-negative-h11} one has
$$ c_1^2\leq \max(5c_2+6b_1,6c_2).$$
\end{corollary}

\begin{remark}
For this remark assume that the characteristic $p\geq 3$. In the
absence of crystalline torsion, $h^{1,1}_W$ detects obstruction to
lifting to $W_2$. More precisely, if $X$ has torsion free
$H^2_{cris}(X/W)$, and $h^{1,1}_W<0$, then $X$ does not lift to
$W_2$.
\end{remark}

\subsection{Examples of Szpiro, Ekedahl}\label{examples-of-szpiro}
As was pointed out in \cite{ekedahl-diagonal} the counter examples
constructed by Szpiro in \cite{szpiro79} also provide examples of
surfaces which are beyond the \eqref{weakBMY} faultline. We
briefly recall these examples. In \cite{szpiro79} Szpiro
constructed examples of smooth projective surfaces $S$ together
with a smooth, projective and non-isotrivial fibration $f:S\to C$
where the fibres has genus $g\geq 2$ and $C$ has genus $q\geq 2$.
Let $f_n: S_n\to C$ be the fibre product of $f$ with the
$n^{th}$-iterate of Frobenius $F_{C/k}: C\to C$. Then
\begin{eqnarray}
c_2(S_n)&=&4(g-1)(q-1)\\
c_1^2(S_n)&=&p^nd+8(g-1)(q-1)
\end{eqnarray}
where $d=deg(f_*(\Omega^1_{X/C}))$ is a positive integer. Thus in
this case, as was pointed in \cite{ekedahl-diagonal},
$h^{1,1}_W\to -\infty$ as $n\to \infty$. Further observe, as $d\geq 1$, that 
\be 
c_1^2>pc_2
\ee
for $n$ large; and also that for any given integer $m\geq 1$, there exists a smooth, projective, minimal surface of general type such that $c_1^2>p^m c_2$.

\subsection{Weak Bogomolov-Miyaoka-Yau inequality holds in characteristic zero}
Assume for this remark that $k=\C$, and that
$X$ is a smooth, projective surface. Then using the Hodge
decomposition for $X$, Noether's formula can be written as
\begin{equation}
h^{1,1}=10\chi(\O_X)-c^2_1+b_1,
\end{equation}
and as the left hand-side of this formula is always non-negative
we deduce that 
\be c_1^2\leq 10\chi(\O_X)+b_1.\ee 
This is easily seen
to be equivalent to 
\be 
c_1^2\leq 5c_2+6b_1.
\ee

\subsection{Lower bounds on $h^{1,1}_W$}\label{lower-bounds}
In this subsection we are interested in lower bounds for
$h^{1,1}_W$. It turns out that unless we are in characteristic
$p\leq 7$, the situation is not too bad thanks to a conjecture of
Raynaud (which is a theorem of Shepherd-Barron).
\begin{proposition}\label{raynaud-lower-bound} Let $X$
be a smooth projective surface of general type. Then
\begin{enumerate}
\item except  when $p\leq 7$ and $X$ is fibred over a curve of
genus at least two and the generic fibre is a singular rational
curve of arithmetic genus at most four we have
\begin{equation}
-c_1^2\leq h^{1,1}_W \leq h^{1,1}.
\end{equation}
\item If $c_2>0$ then $h^{1,1}_W>-\frac{1}{6}c_1^2$.
\item If  $X$ is not uniruled then
\begin{equation}
-\frac{1}{6}c_1^2\leq h^{1,1}_W\leq h^{1,1}.
\end{equation}
\item If  $h^{1,1}_W<-\frac{1}{6}c_1^2$ then there exists
a morphism $X\to C$ with connected fibres and $C$ has genus at
least one.
\end{enumerate}
\end{proposition}
\begin{proof}
We prove(1). Assume if possible that $h^{1,1}_W<-c_1^2$. Then by
using the formula $h^{1,1}_W=b_1+10\chi(\O_X)-c_1^2$ we get
$b_1+10\chi(\O_X)<0$. As $b_1\geq 0$ this implies that
$\chi(\O_X)<0$.  By \cite[Theorem 8]{shepherd-barron91b} we know
that any surface of general type with negative $\chi(\O_X)$ we
have $p\leq 7$; and whenever $\chi(\O_X)<0$ the surface $X$ is
fibred over a curve of genus at least two and the generic fibre is
singular rational curve of genus at most four. Next we prove (2)
and (3) which are really consequence of Raynaud's conjecture which
was proved in \cite{shepherd-barron91b}, using the formula for
$h^{1,1}_W$ in terms of $c_1^2,c_2,b_1$. So suppose that $X$ is
not uniruled and assume, if possible, that
\begin{equation}
h^{1,1}_W<-\frac{1}{6}c_1^2
\end{equation}
Then writing out Ekedahl's formula for $h^{1,1}_W$ we get
\begin{equation}
h^{1,1}_W=b_1+\frac{5}{6}c_2-\frac{1}{6}c_1^2<-\frac{1}{6}c_1^2,
\end{equation}
and so this forces:
\begin{equation}
b_1+\frac{5}{6}c_2<0
\end{equation}
and as $b_1\geq 0$ we see that $c_2$ is negative. Now by
\cite[Theorem~7, page 263]{shepherd-barron91b} we see that $X$ is
uniruled which contradicts our hypothesis. Now we prove (4). This
is a part of the proof  of Raynaud's conjecture in
\cite{shepherd-barron91b}. It is clear that the hypothesis implies
that $c_2<0$.  So by loc. cit. we know that the map $X\to \Alb(X)$
has one dimensional image, and this finishes the proof.
\end{proof}

\begin{remark}
\begin{enumerate}
\item By a result of \cite{lang79}, exceptions in
Theorem~\ref{raynaud-lower-bound}(1) do occur.
\item Thus the examples of surfaces given in
Subsection~\ref{examples-of-szpiro} satisfy the inequality in
Proposition~\ref{raynaud-lower-bound}.
\end{enumerate}
\end{remark}
The following is rather optimistic expectation (because of the paucity of examples) and it would not surprise us if it turns out to be false. 
\begin{conj}\label{negative-h11-fibration}
        If $X$ is a smooth projective surface with 
$-\frac{1}{6}c_1^2\leq h^{1,1}_W<0$ and  $b_1\neq 0$ then the
image of the Albanese map $X\to \Alb(X)$ is one dimensional.
\end{conj}

\subsection{A class of surfaces general type surfaces for which $c_1^2\leq 5c_2$}
Let us begin with the following proposition.
\begin{proposition}\label{m11andpg} Let $X$ be a smooth, projective, minimal surface
of general type
such that
\begin{enumerate}
\item $X$ is
Hodge-Witt,
\item $c_2>0$,
\item $m^{1,1}\geq 2p_g$,
\end{enumerate}
Then $$c_1^2\leq 5c_2$$ holds for $X$.
\end{proposition}
\begin{proof}
Since $X$ is minimal of general type so $c_1^2>0$.
By the formula for $h^{1,1}_W$ we have
$$6h^{1,1}_W=6(m^{1,1}-2T^{0,2})=5c_2-c_1^2+6b_1.$$
As $X$ is Hodge-Witt we see that $T^{0,2}=0$ and so
$$6(m^{1,1}-b_1)=5c_2-c_1^2.$$
Hence the asserted inequality holds if $m^{1,1}-b_1\geq 0$.
Writing $$m^{1,1}-b_1=(m^{1,1}-2p_g)+(2p_g-2q),$$ where we have
used $b_1=2q$. Thus to prove the proposition it will suffice to
prove that each of the two terms in the parenthesis are
non-negative. The first holds by the hypothesis of the proposition
and for the second, we see, as $2q\leq 2h^{0,1}$, that $$2p_g-2q\geq
2p_g-2h^{0,1}=2(\chi(\O_X)-1).$$
Thus to prove the proposition, it suffices to show that 
we have $\chi(\O_X)\geq 1$. This is immediate, from Noether's formula~\ref{noether-formula2} and  our hypothesis that $c_2>0$.
\end{proof} 

\begin{remark} Let us remark that in characteristic $p>0$, $\chi(\O_X)$ may be non-positive and likewise $c_2$ can be non-postive.  However it has been shown in \cite{shepherd-barron91a} that if
$p\geq 7$, then $\chi(\O_X)\geq 0$ for any smooth, projective
minimal surface of general type. Moreover if $\chi(\O_X)=0$, then $c_2<0$
by Noether's formula \ref{noether-formula1}. It was shown in 
\cite{shepherd-barron91b}, if $c_2<0$ then $X$ is uniruled (and in any case if
$c_2<0$, then the inequality $c_1^2\leq 5c_2$ is false). In contrast if $k=\C$, a well-known result of Castelnouvo says $c_2\geq 0$ and $\chi(\O_X)>0$ ($X$ minimal of general type).
\end{remark}

The following proposition is a variant of Proposition~\ref{m11andpg} and is valid for the larger class of Mazur-Ogus surfaces. This proposition gives a sufficient condition (in terms of slopes of Frobenius and the geometric genus of the surface) for $c_1^2\leq 5c_2$ to hold.

\begin{proposition}\label{m11andpg2} Let $X$ be a smooth, projective, minimal surface
of general type
such that
\begin{enumerate}
\item $X$ is Mazur-Ogus,
\item $c_2>0$,
\item $m^{1,1}\geq 4p_g$,
\end{enumerate}
Then $$c_1^2\leq 5c_2$$ holds for $X$.
\end{proposition}

\begin{proof}
We argue as in the proof of Proposition~\ref{m11andpg}. Since $X$ is minimal of general type so $c_1^2>0$. By the formula for $h^{1,1}_W$ we have
$$6h^{1,1}_W=6(m^{1,1}-2T^{0,2})=5c_2-c_1^2+6b_1.$$
As $X$ is Mazur-Ogus we see that $h^{1,1}_W=h^{1,1}$
and so we have
$$6(m^{1,1}-2T^{0,2}-b_1)=5c_2-c_1^2.$$
Hence the asserted inequality holds if $m^{1,1}-2T^{0,2}-b_1\geq 0$.
Writing $$m^{1,1}-2T^{0,2}-b_1=(m^{1,1}-4p_g)+(2p_g-2T^{0,2})+(2p_g-2q),$$
where we have
used $b_1=2q$. Thus to prove the proposition it will suffice to
prove that each of the three terms in the parentheses are
non-negative. The first holds by the hypothesis of the proposition
and for the second we argue as follows:
as $X$ is Mazur-Ogus, so
$$b_2=h^{0,2}+h^{1,1}+h^{2,0}=h^{1,1}+2p_g$$
by degeneration of Hodge-de Rham at $E_1$.
Further
$$b_2-2p_g=h^{1,1}=h^{1,1}_W=m^{1,1}-2T^{0,2}.$$
Hence
$$b_2-2p_g=m^{1,1}-2T^{0,2}.$$
So we get
$$b_2-m^{1,1}=2(p_g-T^{0,2}).$$
Now
$$b_2=m^{0,2}+m^{1,1}+m^{2,0},$$
which shows that $b_2-m^{1,1}\geq 0$ and so $p_g-T^{0,2}\geq 0$. Hence this term is non-negative. For the third term we see, as $2q\leq 2h^{0,1}$, that $$2p_g-2q\geq
2p_g-2h^{0,1}=2(\chi(\O_X)-1).$$ Thus the proposition  follows as
as we have $\chi(\O_X)\geq 1$  from our hypothesis that $c_2>0$ and Noether's formula~\ref{noether-formula2}.
\end{proof}

We do not know how often the inequality $m^{1,1}\geq 4p_g$ holds. But the following Proposition shows that for surfaces of large degree in $\P^3$ the inequality $m^{1,1}\geq 4p_g$ holds.

\begin{proposition}\label{chern-example}
Let $X\subset \P^3$ be a smooth hypersurface of degree $d$. Then if $d\geq 5$, $X$ satisfies all the hypothesis of Proposition~\ref{m11andpg2}. Hence the class of surfaces to which Proposition~\ref{m11andpg2} applies is non-empty.
\end{proposition}
\begin{proof}
Let us assume $X$ is a smooth, projective surface of degree $d$ in $\P^3$. From the formulae for Hodge numbers in \cite{rapoport72,lewis-survey} it is clear for smooth, projective surfaces in $\P^3$, the numbers $b_2,h^{1,1},h^{0,2}=h^{2,0}=p_g$ depend only on $d$ and are constant in the family of smooth surfaces. Further Hodge-de Rham spectral sequence for $X$ degenerates at $E_1$ and crystalline cohomology of $X$ is torsion-free. Since $b_1=0$ we see that $c_2>0$. Thus all the hypothesis of Proposition~\ref{m11andpg2} are satisfied except possibly the hypothesis that $m^{1,1}\geq 4p_g$.

From the formula for Hodge-Witt numbers we have
$$h^{1,1}=h_W^{1,1}=m^{1,1}-2T^{0,2},$$
and as $h^{1,1}$ is constant in this family (as it depends only on the degree), while $T^{0,2}$ can only increases under specialization (this is a result of Richard Crew~\cite{crew85}), so we see that $m^{1,1}$ must also increase under specialization. At any rate we have $h^{1,1}\leq m^{1,1}$. So it suffices to prove that $h^{1,1}\geq 4p_g$.  Now by Table~\ref{hodge-table}, we see that $h^{1,1}=b_2-2p_g$
and simple calculation shows
\be 
	h^{1,1}-4p_g = 2d^2 - 5d + 4
\ee
which is positive for $d\geq 5$. Hence for $d\geq 5$, 
\be 
	m^{1,1}\geq h^{1,1}>4p_g>2p_g.
\ee
\end{proof}

\begin{remark}
It seems reasonable to expect that for large $c_1^2,c_2$, in the moduli of smooth, projective  surfaces of general type, there is a Zariski open set of an irreducible component(s) consisting of Mazur-Ogus surfaces  where $m^{1,1}\geq 4p_g$ holds.
\end{remark}

The following result, while not best possible, shows how we can use preceding ideas to obtain a Chern class inequality under reasonable geometric hypothesis. We do not know a result of comparable strength which can be proved by purely geometric means.

\begin{theorem}\label{strange-thm}
Let $X$ be a smooth, projective, minimal surface of general type. Assume
\begin{enumerate}
\item $c_2>0$,
\item $p_g>0$,
\item $X$ is Hodge-Witt,
\item $\Pic(X)$ is reduced or $H^2_{cris}(X/W)$ is torsion free,
\item and $H^2_{cris}(X/W)$ has no slope $<\frac{1}{2}$.
\end{enumerate}
Then
\bes
c_1^2\leq 5c_2.
\ees
\end{theorem}

\begin{remark}
Before proving the theorem let us remark that in a smooth family of surfaces the Hodge-Witt locus is open (this is a result of Crew, see \cite{crew85}); the slope condition in our hypothesis is a closed condition in Newton strata (Newton polygon of $X$ lies on or above a finite list of polygons). So these two conditions provide a locally closed subset of base space. If we consider moduli of surfaces with fixed $c_1^2, c_2$, then by \cite{liedtke09} if $p$ is  larger than a constant depending only on $c_1^2$, then $\Pic(X)$ is reduced. So our hypothesis while not generic in the moduli are relatively harmless.
\end{remark}

\begin{proof}[Proof of Theorem~\ref{strange-thm}]
Before proceeding let us make two remarks. Firstly the assumption that the slopes of Frobenius on $H^2_{cris}(X/W)$ are $\geq \frac{1}{2}$, together with the assumption that $p_g>0$ says that $X$ is Hodge-Witt but not ordinary (by Mazur's proof of Katz's conjecture); secondly we see that $\chi(\O_X)\geq 1$. This follows from Noether's formula~\ref{noether-formula1}, and the fact that for $X$  minimal of general type,  $c_1^2\geq 1$. Observe that $c_2>0$ is necessary for $c_1^2\leq 5c_2$ to hold (as $c_1^2\geq 1$). Thus this hypothesis (that $c_2>0$) is at any rate required if we wish to consider Chern class inequality of Bogomolov Miyaoka type to hold. Moreover our hypothesis that $H^2_{\rm cris}(X/W)$ is torsion-free implies $\Pic(X)$ is reduced (see \cite[Remark 6.4, page 641]{illusie79b}).

Proposition~\ref{m11andpg} shows that we have to prove that
$m^{1,1}\geq 2p_g$ (under our hypothesis). We now prove this inequality under our
hypothesis $\Pic(X)$ is reduced and $X$ is Hodge-Witt and $H^2_{cris}(X/W)$ satisfies the  stated slope condition.

We begin by recalling the formula for $m^{1,1}$ (see \ref{slope-number-definition}).
\begin{equation}
m^{1,1}=\sum_{\lambda\in[0,1[}\lambda
m_\lambda+\sum_{\lambda\in[1,2[}(2-\lambda)m_\lambda,
\end{equation}
which, on writing $m_0=\dim H^2_{cris}(X/W)_{[0]}$, $m_1=\dim
H^2_{cris}(X/W)_{[1]}$, and noting that $m_0$ does not contribute
to $m^{1,1}$, can be written as
\begin{equation}
m^{1,1}=\sum_{\lambda\in]0,1[}\lambda
m_\lambda+m_1+\sum_{\lambda\in]1,2[}(2-\lambda)m_\lambda.
\end{equation}

Poincar\'e duality says that if $\lambda$ is a slope of $H^2_{\rm cris}(X/W)$ then $\lambda'=2-\lambda$ is also a slope of $H^2_{\rm cris}(X/W)$. Further under Poincar\'e duality we get $m_\lambda=m_{\lambda'}$. Hence for any $\nu\in]1,2[$ we  have
\be
(2-\nu)m_{\nu}=(2-\nu)m_{\nu'}=\nu' m_{\nu'},
\ee
and for any $\nu\in]1,2[$, $\nu'\in]0,1[$. Thus we see that
\be\sum_{\lambda \in]0,1[}\lambda m_{\lambda}+\sum_{\lambda\in]1,2[}(2-\lambda)m_{\lambda}=2\sum_{\lambda}\lambda m_{\lambda}
\ee
Thus we get
\bes
m^{1,1}=m_1+2\sum_{\lambda\in]0,1[}\lambda m_\lambda.
\ees
Under out hypothesis we claim that the  sum on the right in the above equation is $\geq 2p_g$.
This
will involve the assumption that $\Pic(X)$ is reduced and the assumption that $X$ is Hodge-Witt. The assumption that $X$ is Hodge-Witt says that $H^2(W(\O_X))$ is a finite type $W$-module (see \cite{illusie79b}).   Our hypothesis that $\Pic(X)$ is reduced means $V$ is injective on $H^2(W(\O_X))$. We claim that $H^2(W(\O_X))$ is a free, finite type $W$-module. If $H^2_{cris}(X/W)$ is torsion-free, then this is an immediate consequence of the existence of the Hodge-Witt decomposition of $H^2_{cris}(X/W)$ (see \cite[Theorem 4.5, page 202]{illusie83b}):
\be
H^2_{cris}(X/W)=H^2(X,W(\O_X))\oplus H^1(X,W\Omega^1_X)\oplus H^0(X,W\Omega_X^2),
\ee
which implies that $H^2(W(\O_X))$ is torsion-free.

Now suppose instead that $\Pic(X)$ is reduced. We want to prove that $H^2(X,W(\O_X))$ is torsion-free. We note that $X$ is Hodge-Witt so we see that $H^2(X,W(\O_X))$ is finite type as a $W$-module and profinite (as a $W[F,V]$-module). Now $\Pic(X)$ is reduced, by \cite[Remark 6.4,page 641]{illusie79b}, gives us the injectivity of $V$ on $H^2(X,W(\O_X))$. Hence $H^2(X,W(\O_X))$ is a Cartier module (see \cite[Def. 2.4]{illusie83b}). So by \cite[Proposition 2.5(d), page 99]{illusie83b}  $H^2(X,W(\O_X))$ is free of finite type over $W$.  In particular we have an exact sequence of
\be
\xymatrix{
0\ar[r]& H^2(W(\O_X))\ar[r]^V &H^2(W(\O_X))\ar[r]& H^2(\O_X)\ar[r]& 0.}
\ee
which shows that the
$W$-rank of the former is at least $p_g>0$.

Let us remind the reader that the degeneration of the slope spectral sequence modulo torsion (see \cite{illusie79b}) or the existence of the Hodge-Witt decomposition of $H^2_{cris}(X/W)$ (as above) shows that the slopes $0\leq \lambda<1$ of $H^2_{cris}(X/W)$ live in $H^2(W(\O_X))$.

Next note that for $\lambda$ satisfying
\be
\frac{1}{2}\leq\lambda< 1
\ee
we have $$\lambda\geq 1-\lambda>0.$$ Thus we have
\be
\sum_{\lambda\in]0,1[}\lambda m_\lambda \geq \sum_{\lambda\in]0,1[} (1-\lambda) m_\lambda.
\ee
We claim now that
\be 
\sum_{\lambda\in]0,1[} (1-\lambda) m_\lambda = \dim H^2(W(\O_X))/VH^2(W(\O_X)).
\ee This is standard and is a consequence of the proof of \cite[Lemma 3]{crew85}--we state Crew's  result explicitly as  Lemma~\ref{crew-lemma} (below) for convenient reference. Finally the above exact sequence shows that
\be H^2(W(\O_X))/VH^2(W(\O_X))\isom H^2(\O_X),
\ee and hence that the sum is $\geq p_g>0$, and therefore
\be
m^{1,1}\geq m_1+2p_g>2p_g,
\ee
as $m_1\geq 1$ (by projectivity of $X$) so we are done by Proposition~\ref{m11andpg}.
\end{proof}

\begin{lemma}\cite[Lemma 3]{crew85}\label{crew-lemma}
Let $M$ be a $R^0=W[F,V]$-module (with $FV=p$) and such that $M$ is finitely generated and free as a $W$-module. Assume that the slopes of Frobenius on $M$ satisfy $0\leq\lambda<1$. Then
\bes 
\textrm{length}(M/VM)=\sum_{\lambda}(1-\lambda)m_\lambda.
\ees
\end{lemma}

\begin{remark}
Let us give examples of surfaces of general type which satisfy the hypothesis of our theorem. In general construction of Hodge-Witt but non-ordinary surfaces is difficult. Suppose $C,C'$ are smooth, proper curves over $k$ (perfect) with genus $\geq 2$, and that $C$ is ordinary and $C'$ has slopes of Frobenius $\geq 1/2$ in $H^1_{cris}(C'/W)$. Such curves exist--for example there exist curves of genus $\geq 2$ whose Jacobian is a supersingular abelian variety. Now let $X=C\times_k C'$. Then by a well-known theorem of Katz and Ekedahl \cite[Proposition 2.1(iii)]{ekedahl85}, $X$ is Hodge-Witt and by Kunneth formula, the slopes of Frobenius on $H^2_{cris}(X/W)$ are $\geq 1/2$. The other hypothesis of Theorem~\ref{strange-thm} are clearly satisfied.
\end{remark}

\subsection{On $c_1^2\leq 5c_2$ for supersingular surfaces}
In this subsection, we will say that $X$ is a \textit{supersingular surface} if $H^2(X,W(\O_X))\tensor_W K=0$ (here $K$ is the quotient field of $W$). This means that $H^2_{cris}(X/W)$ has no slopes in $[0,1)$ and hence by Poincar\'e duality, it has no slopes in $(1,2]$. Thus $H^2_{cris}(X/W)$ is pure slope one. Under reasonable assumptions the following dichotomy holds:
\begin{proposition}\label{chern-supersingular}
Let $X$ be a smooth, projective surface over a perfect field of characteristic $p>0$. Assume
\begin{enumerate}
\item $X$ is a minimal surface of general type,
\item $p_g>0$
\item $c_2>0$,
\item $X$ is Mazur-Ogus,
\item $X$ is supersingular.
\end{enumerate}
Then either
$$c_1^2\leq 5c_2,$$
or
$$c_2<2\chi(\O_X),$$ and in the second case no smooth deformation of $X$ admits any flat lifting to characteristic zero.
\end{proposition}
\begin{proof}
Since $X$ is Mazur-Ogus (i.e. Hodge-de Rham spectral sequence of $X$ degenerates at $E_1$, and crystalline cohomology of $X$ is torsion-free), and  supersingular,  we see that
\bea
b_2&=&h^{2,0}+h^{1,1}+h^{0,2},\\
m^{1,1}&=&b_2.
\eea
Thus \ref{crew-formula} gives
\be
h^{1,1}_W=h^{1,1}=b_2-2p_g=m^{1,1}-2T^{0,2}=b_2-2T^{0,2},
\ee
which gives $T^{0,2}=p_g>0$, so $X$ is not Hodge-Witt. Further we have
\be
h^{1,1}_W=h^{1,1}=b_2-2p_g=\frac{5c_2-c_1^2}{6}+b_1.
\ee
This gives
\be
6(b_2-b_1-2p_g)=5c_2-c_1^2.
\ee
Thus we $$c_1^2\leq 5c_2$$ if and only if $$b_2-b_1-2p_g\geq 0.$$
If $b_2-b_1-2p_g\geq 0$ then $c_1^2\leq 5c_2$ and the first assertion holds and we are done.
Now suppose $b_2-b_1-2p_g<0$. So we get
$$
	b_2<b_1+2p_g.
$$
Now we get from the fact that
\be
c_2=b_2-2b_1+2,
\ee
so that
\be
c_2 = b_2-2b_1+2< b_1+2p_g-2b_1+2=2p_g-b_1+2=2\chi(\O_X).
\ee
Hence
\be
c_2<2\chi(\O_X).
\ee
On the other hand note that if $k=\C$ and $X/\C$ is smooth, minimal of general type then $c_2\geq 3\chi(\O_X)$ (by the Bogomolov-Miyaoka-Yau inequality and Noether's formula) so $c_2<2\chi(\O_X)$ never happens over $\C$. Hence no smooth deformation of a surface with $c_2<2\chi(\O_X)$ can be liftable to characteristic zero (as $c_2,\chi(\O_X)$ are deformation invariants).
\end{proof}

\begin{remark}
We do not know if the condition (on   minimal surfaces of general type) that $c_2<2\chi(\O_X)$ is relatively rare or even bounded.
\end{remark}

\subsection{A lower bound on slopes of Frobenius}
In this section we prove a lower bound on slopes Frobenius on $H^2_{cris}(X/W)$ for  smooth, projective surfaces. This theorem shows that the assumption on the smallest slope of Frobenius in Theorem~\ref{strange-thm} is perhaps not too unreasonable. The theorem is the following:
\bthm\label{th:lower-bound-for-slopes}
Let $X/k$ be a smooth, projective surface over a perfect field $k$ of characteristic $p>0$. Let 
$$\lambda^2_{\min}=\min\{\lambda: \lambda \text{ is a slope of Frobenius on } H^2_{cris}(X/W)\}.$$  Assume that  $p_g\neq 0$.
Then exactly one of the following holds:
\benum
\item either $\lambda^2_{\min}=0$, or
\item $\lambda^2_{\min}\geq \frac{1}{p_g+1}$.
\eenum
\ethm

\begin{remark}
In \cite{mazur73} it was shown that if $X$ is a smooth, projective surface with $p_g=1$ and $X$ has torsion free crystalline cohomology then $$\lambda^2_{\min}=\begin{cases}
0 & \textrm{or},\\
1-\frac{1}{n}& \textrm{with } n\geq 2.
\end{cases}$$  and so if $\lambda^2_{\min}\neq0$ one has $\lambda^2_{\min}\geq \frac{1}{2}$.
\end{remark}

\begin{remark}
While Theorem~\ref{th:lower-bound-for-slopes} asserts that $\lambda_{\min}\geq \frac{1}{p_g+1}$,  in Theorem~\ref{strange-thm} we had assumed that $\lambda_{\min}\geq \frac{1}{2}$. It is tempting to hope that perhaps surfaces of general type with $\frac{1}{p_g+1}\leq\lambda_{\min}<\frac{1}{2}$ do not occur or occur in bounded families (for each fixed characteristic). But again we have no evidence for this if $p_g>1$.
\end{remark}

\bp 
If $\lambda^2_{\min}=0$ then there is nothing to prove. So assume $\lambda^2_{\min}>0$. 

Recall that $m^{0,2}=\sum_{\lambda\in [0,1[}(1-\lambda)m_\lambda$, and 
\be\label{th:lower-bound-eq-1} 
0\leq m^{0,2}\leq m^{0,2}+T^{0,2}=h^{0,2}_W\leq h^{0,2}=p_g\neq 0.
\ee
For notational convenience let $\nu=\lambda^2_{\min}$ and let $N=p_g+1$. Suppose if possible that $0<\nu<\frac{1}{N}$.

Then $\frac{1}{\nu}>N$ and $-\nu > -\frac{1}{N}$ so $1-\nu > 1-\frac{1}{N}$. Thus one has
$$m_\nu(1-\nu) > m_\nu\left(1-\frac{1}{N}\right).$$
Now $\nu m_\nu\geq 1 $ (as $0<\nu m_\nu\in\Z$ by Dieudonne Theory) so 
$$m_\nu \geq \frac{1}{\nu}>N.$$
Hence one sees that 
$$m_\nu(1-\nu)>m_\nu\left(1-\frac{1}{N}\right)>\left(1-\frac{1}{N}\right)N=N-1=p_g,$$
in particular 
\be\label{th:lower-bound-eq-2}
 m_\nu(1-\nu) >p_g.
\ee
Thus one has combining \eqref{th:lower-bound-eq-1} and \eqref{th:lower-bound-eq-2} that
$$p_g< (1-\nu)m_\nu \leq m^{0,2} \leq p_g$$
which is a contradiction. Thus $\nu\geq \frac{1}{N}=\frac{1}{p_g+1}$.
\ep

\subsection{Recurring Fantasy for ordinary surfaces}\label{recurringfantasy}
In this subsection we sketch a very optimistic conjectural  program (in fact we are still somewhat reluctant to call it a conjecture--perhaps, following Spencer Bloch, it would be better to call it a \textit{recurring fantasy}) to prove the analog of van de Ven's inequality \cite{deven76a} for ordinary surfaces of general type and which satisfy Assumptions~\ref{assumptions}. Unfortunately we do not know how to prove our conjecture (see Conjecture~\ref{ordinaryconj} below). We will make the following assumptions on $X$ smooth, projective over an algebraically closed field $k$ of characteristic $p>0$.
\subsubsection{Assumptions}\label{assumptions}
For the entire Subsection~\ref{recurringfantasy} we make the following assumptions on a smooth, projective surface $X$:
\begin{enumerate}
 \item $X$ is minimal of general type,
 \item $X$ is not fibred  over a smooth, projective curve of genus $g>1$,
 \item $X$ is an ordinary surface,
 \item and $X$ has torsion-free crystalline cohomology.
\end{enumerate}
The last two assumptions imply (see \cite{illusie83b}) that
\begin{enumerate}
 \item Hodge and Newton polygons of $X$ coincide and,
 \item We have a Newton-Hodge decomposition:
	$$H^n_{cris}(X/W)=\oplus_{i+j=n}H^i(X,W\Omega^j_X),$$
\item the Hodge-de Rham spectral sequence of $X$ degenerates at $E_1$.
\end{enumerate}

In particular we have
\be
H^1_{cris}(X/W)=H^0(X,W\Omega^1_X)\oplus H^1(X,W\O_X),
\ee
and
\be{\rm rk}_W H^0(X,W\Omega^1_X)={\rm rk}_W H^1(X,W\O_X).
\ee

By the usual generalities (see \cite{illusie79b}) we have a cup product pairing of $W[F,V]$-modules (here $FV=p$):
\begin{equation}\label{cupproduct}
<\ ,\ >:H^1(X,W\O_X)\tensor H^0(X,W\Omega^1_X) \to H^1(X,W\Omega^1_X).
\end{equation}

\begin{conj}\label{ordinaryconj}
 For any surface $X$ satisfying assumptions of (1)--(4) of section~\ref{assumptions}), the cup product paring of \eqref{cupproduct} satisfies the following properties:
\begin{enumerate}
 \item for each fixed $0\neq v\in H^1(X,W\O_X)$, the map $<v,->$ is injective;
 \item for each fixed $0\neq v'\in H^0(X,W\Omega^1_X)$, the mapping $<-,v'>$ is injective.
\end{enumerate}
\end{conj}

\begin{corollary}
Assume Conjecture~\ref{ordinaryconj} and let
\begin{eqnarray}
      h^{0,1}&=&{\rm rk}_W H^1(X,W\O_X),\\
      h^{1,0}&=&{\rm rk}_W H^0(X,W\Omega_X^1),\\
      h^{1,1}&=&{\rm rk}_W H^1(X,W\Omega_X^1),
\end{eqnarray}
Then
\bes
	h^{1,1}\geq 2h^{1,0}-1=b_1-1.
\ees
\end{corollary}

It is clear that, assuming Conjecture~\ref{ordinaryconj}, this can be proved in a manner similar to van de Ven's proof of the above inequality (see \cite{deven76a}). The conjecture and the above inequality has the following immediate consequence.
\begin{theorem}\label{ordconjthm}
Under the assumptions of \ref{assumptions} and conjecture~\ref{ordinaryconj} on $X$, the Chern classes of $X$ satisfy
\bes c_1^2\leq 5c_2+6.
\ees
\end{theorem}
\begin{proof}
 Recall the formula of Crew and Ekedahl \ref{crew-formula}
\begin{equation}
 6h^{1,1}_W=6(m^{1,1}-T^{0,2})=5c_2-c_1^2+6b_1,
\end{equation}
where $h^{1,1}_W$ is the Hodge-Witt number of $X$ and $m^{1,1}$ is the slope number of $X$.
If $X$ is ordinary, we see that $T^{0,2}=0$ and so $m^{1,1}=h^{1,1}$ is the dimension of the slope one part of $H^2_{cris}(X/W)$.

Hence $$5c_2-c_1^2+6b_1=6m^{1,1}=6h^{1,1},$$ so that
$$5c_2-c_1^2=6m^{1,1}-6b_1=6h^{1,1}-6b_1.$$ Now by the corollary we have
$$h^{1,1}\geq b_1-1$$
so that
\be
	h^{1,1}-b_1\geq -1
\ee
and so
\be
	5c_2-c_1^2\geq -6,
\ee
or equivalently,
\be
	c_1^2\leq 5c_2+6.
\ee
This completes the proof.
\end{proof}

\begin{theorem}
Assume Conjecture~\ref{ordinaryconj}. Then except for a bounded family of surfaces satisfying \ref{assumptions}, we have
$$c_1^2\leq 6c_2.$$
\end{theorem}
\begin{proof}
By Theorem~\ref{ordconjthm}, $c_1^2\leq 5c_2+6$ holds for all surfaces satisfying \ref{assumptions}. Now consider surfaces for which \ref{assumptions} hold and $c_2< 6$. Then, for these surfaces $c_1^2\leq 5c_2+6<36$. Thus surfaces which satisfy $c_2<6$ also satisfy $c_1^2<36$ (under \ref{assumptions}). Now surfaces of general type which satisfy $c_1^2<36$ and  $c_2<6$ form a bounded family. For surfaces which do not belong to this family $c_2\geq 6$. Hence
\be
c_1^2\leq 5c_2+6<5c_2+c_2=6c_2.
\ee
\end{proof}

\section{Enriques classification and Torsion in crystalline
cohomology}\label{geography-of-crystalline-torsion} The main aim
of this section is to explore geographical aspects of torsion in
crystalline cohomology. It is well-known that if $X/\C$ is a
smooth, projective surface then the torsion in $H^2(X,\Z)$ is
invariant under blowups. We will see a refined version  of this
result holds in positive characteristic (see
Theorem~\ref{torsion-species-and-blowups}). One of the main results of this section provides 
 a new birational invariant of smooth surfaces in characteristic $p>0$.

\subsection{Gros' Blowup Formula} In the next few subsections we will use the formulas
which describe the behavior of cohomology of the de Rham-Witt
complex under blowups. We recall these from \cite{gros85}. Let $X$
be a smooth projective variety and let $Y\subset X$ be a closed
subscheme, pure of codimension $d$. Let $X'$ denote the blowup of
$X$ along $Y$, and let $f:X'\to X$ be the blowing up morphism.
Then one has an isomorphism:
\begin{equation}\label{blowup-formula}
\xymatrix{
{H^j(X,W\Omega^i_X)\displaystyle{\bigoplus_{0<n<d}} H^{j-n}(Y,W\Omega_Y^{i-n})} 
\ar[r]^<<{\sim} & H^j({X'},W\Omega^i_{X'}).}
\end{equation}

\subsection{Birational invariance of the domino of a surface}
The de Rham-Witt cohomology of a smooth, projective surface has only one, possibly
non-trivial, domino. This is the domino associated to the
differential $H^2(X,W(\O_X))\to H^2(X,W\Omega^1_X)$. In this
section we prove the following.
\begin{theorem}\label{birational-domino}
Let $X,X'$ be two smooth, projective surfaces over an
algebraically closed field $k$ of characteristic $p>0$ and let
$X'\to X$ be a birational morphism. Then the dominos $\dom^{0,2}(X)$ (resp. $\dom^{0,2}(X')$) associated to
the differentials $H^2(X,W(\O_X))\to H^2(X,W\Omega^1_X)$  (resp. $H^2(X,W(\O_X))\to H^2(X',W\Omega^1_{X'})$ are
naturally isomorphic.
\end{theorem}
\subsection{Proof of Theorem~\ref{birational-domino}}
As any birational morphism $X'\to X$ as above factors as a finite
sequence of blowups at closed points, we may assume that $X'\to X$
is the blowup of $X$ at a single point. In what follows, to
simplify notation, we will denote objects on $X'$ by simply
writing them as \textit{primed quantities} and the \textit{unprimed quantities} will
denote objects on $X$. We will use the notation of
Subsection~\ref{differential0}.

The construction of the de Rham-Witt complex $W\Omega^\mydot_X$ is
functorial in $X$. The properties of the de Rham-Witt complex (in
the derived category of complexes of sheaves of modules over the
Cartier-Dieudonne-Raynaud algebra) under blowing up have been
studied extensively in \cite{gros85}, and using \cite[Chapter 7,
Theorem 1.1.9]{gros85}, and the usual formalism of de Rham-Witt
cohomology, we also have a morphism of slope spectral sequences.
The blowup isomorphisms described in the blowup formula fit into
the following diagram
$$\xymatrix{
  H^2(W(\O_{X'})) \ar[d]_{} \ar[r]^{d'} & H^2(W\Omega^1_{X'}) \ar[d]_{} \ar[r]^{d'} & H^2(W\Omega^2_{X'}) \ar[d]^{} \\
  H^2(W(\O_{X})) \ar[r]^{d} & H^2(W\Omega^1_{X}) \ar[r]^{d} &  H^2(W\Omega^2_{X}).  }
$$
By the Gros' blowup formula \ref{blowup-formula} all the vertical arrows
are isomorphisms. This induces an isomorphism
$Z'=\ker(d')\to\ker(d)=Z$.

Now  the formula for blowup for cohomology of the de Rham-Witt
complex also shows that we have isomorphisms for $i=1,2$,
$$\xymatrix{ H^i(X',W(\O_{X'}))\ar[r]^\simeq & H^i(X',W(\O_{X}))}$$
and these fit into the following commutative diagram.
$$\xymatrix{\small
  H^1(W(\O_{X'})) \ar[d]_{} \ar[r]^{} & H^1(X',\O_{X'}) \ar[d]_{} \ar[r]^{} &
  H^2(W(\O_{X'})) \ar[d]_{}
  \ar[r]^{} & H^2(W(\O_{X'})) \ar[d]_{} \ar[r]^{} & H^2(\O_{X'}) \ar[d]^{} \\
  H^1(W(\O_{X}))  \ar[r]^{} & H^1(X,\O_{X}) \ar[r]^{} &
  H^2(W(\O_{X}))\ar[r]^{} & H^2(W(\O_{X}))  \ar[r]^{} & H^2(\O_{X})
  }$$
In the above commutative diagram we claim that all the vertical
arrows are isomorphisms. Indeed \cite[Proposition~3.4,
V.5]{hartshorne-algebraic} shows that $H^i(X',\O_{X'})\to
H^i(X,\O_{X})$ are isomorphisms for $i\geq0$. The other vertical
arrows are isomorphisms by \cite{gros85}.  Thus from the diagram
we deduce an induced isomorphism
$$\xymatrix{\ker(H^2(W(\O_{X'}))\ar[r]^{V} & H^2(W(\O_{X'})))
\ar[r]^\simeq & \ker(H^2(W(\O_{X}))\ar[r]^{V} & H^2(W(\O_{X})))}.
$$
Thus the $V$-torsion in $H^2(W(\O_X))$ of $X$ and $X'$ in
$H^2(W(\O_{X'}))$ are isomorphic (we will use this in the proof of
our next theorem as well).

Now these two arguments combined also give the corresponding
assertions for the composite maps $dV^n$ (resp. $d'{V'}^n)$. Thus
we also have from a similar commutative diagram (with $dV^n$ etc.)
from which we deduce that we have isomorphisms
$\ker(d'{V'}^n)={V'}^{-n}Z'\to {V}^{-n}Z=\ker(d{V}^n)$. Thus we
have an isomorphism of the intersection of
$${V'}^{-\infty}Z'=\cap_n
{V'}^{-n}Z' \to {V'}^{-\infty}Z'=\cap_n {V}^{-n}Z.$$ Thus we have
in particular, isomorphisms
$$
\frac{H^2(X',W(\O_{X'}))}{{V'}^{-\infty}Z'} \simeq
\frac{H^2(X',W(\O_X))}{{V}^{-\infty}Z}.
$$
Now in the canonical factorization of $d$ (resp. $d'$) in terms of
their dominos we have a commutative diagram
$$\xymatrix{
  H^2(W(\O_{X'}))\ar[d]_{} \ar[r]^{} & \frac{H^2(W(\O_{X'}))}{{V'}^{-\infty}Z'} \ar[d]_{} \ar[r]^{} & {F'}^\infty B' \ar[d]_{}
  \ar[r]^{} & H^2(W\Omega^1_{X'}) \ar[d]^{} \\
  H^2(W(\O_{X})) \ar[r]^{} & \frac{H^2(W(\O_{X}))}{V^{-\infty}Z} \ar[r]^{} & {F}^\infty {B} \ar[r]^{} & H^2(W\Omega^1_{X}).}
$$
The first two vertical arrows and the last are isomorphisms. Hence
so is the remaining arrow. This completes  the proof of the
theorem.

\subsection{$T^{0,2}$ is a birational invariant}
The following corollary is immediate from the above theorem, but we also provide a simple and direct proof of this fact using properties of numerical invariants.

\begin{corollary}\label{birational-invariance-of-domino}
	Let $X,X'$ be smooth, projective surfaces over a perfect field and suppose that $X'\to
	X$ is a birational morphism. Then $T^{0,2}(X)=T^{0,2}(X')$.
\end{corollary}
\begin{proof}[Another proof]
It is enough to assume that the ground field is algebraically closed.	
Using the fact that any birational morphism $X'\to X$ of surfaces
	factors as finite sequence of blowups at closed points, we reduce
	to proving this assertion for the case when $X'\to X$ is the
	blowup at one closed point.
	
	As $c_2$ increases by $1$ and $c_1^2$ decreases by $1$ under
	blowups, the formula for $h^{1,1}_W$ shows that
	$h^{1,1}_W(X')=h^{1,1}_W(X)+1$  while using the formula for
	blowups for crystalline cohomology and a slope computation shows
	that the slope numbers of $X'$ and $X$ satisfy
	$$m^{1,1}(X')=m^{1,1}(X)+1, $$
	here the ``1'' is the contribution coming from the cohomology in
	degree two of the exceptional divisor which is one dimensional, so
	the result follows as
	$$h^{1,1}_W=m^{1,1}-2T^{0,2}.$$
\end{proof}

\subsection{Crystalline Torsion}
We begin by quickly recalling Illusie's results about crystalline
torsion. By crystalline torsion we will mean torsion in the
$W$-module $H^2_{cris}(X/W)$, which we will denote by
$H^2_{cris}(X/W)_{\rm Tor}$. Let $X/k$ be a smooth projective
variety. According to  \cite{illusie79b}, torsion in
$H^{cris}(X/W)$ arises from several different sources (see
\cite[Section 6]{illusie79b}). Torsion in the Neron-Severi group
of $X$, denoted $NS(X/k)_{\rm Tor}$ in this paper, injects into
$H^2_{\cris}(X/W)$ via the crystalline cycle class map (see
\cite[Proposition 6.8, page 643]{illusie79b}). The next species of
torsion one finds in the crystalline cohomology of a surface is
the $V$-torsion, denoted by $H^2_{cris}(X/W)_{v}$. It is the
inverse image of $V$-torsion in $H^2(X,W(\O_X))$, denoted here by
$H^2(X,W(\O_X))_{\rm V-tors}$, under the map $H^2(X/W)\to
H^2(X,W(\O_X))$. It is disjoint from the Neron-Severi torsion (see
\cite[Proposition 6.6, page 642]{illusie79b}). Torsion of these
two species is collectively called the {\em divisorial torsion} in
\cite[page 643]{illusie79b} and denoted by $H^2_{\cris}(X/W)_d$.
The quotient
$$H^2_{\cris}(X/W)_e=H^2_{\cris}(X/W)_{\rm
Tor}/H^2_{\cris}(X/W)_d$$ is called the {\em exotic torsion} of
$H^2_{\cris}(X/W)$, or if $X$ is a surface then simply by the
exotic torsion of $X$.

\subsection{Torsion of all types is invariant under blowups} Our next result concerns the torsion in the
second crystalline cohomology of a surface.
\begin{theorem}\label{torsion-species-and-blowups} Let $X'\to X$
be a birational morphism of smooth projective surfaces. Then
\begin{enumerate}
\item we have an isomorphism
$$H^2_{\cris}(X/W)_{Tor}\to H^2_{cris}(X'/W)_{Tor},$$
\item and this isomorphism induces an isomorphism on the Neron-Severi, the $V$-torsion, and the exotic torsion.
\end{enumerate}
\end{theorem}
\begin{proof}
As every $X'\to X$ as in the hypothesis factors as a finite
sequence of blowups at closed points, it suffices to prove the
assertion for the blowup at one closed point. So let $X'\to X$ be
the blowup of $X$ at one closed point $x\in X$. The formula for
blowup for crystalline cohomology induces an isomorphism
$$\xymatrix{
H^2_{cris}(X/W)_{Tor}\ar[r]^{\simeq} & H^2(X'/W)_{\rm Tor}.}
$$
This proves assertion (1). As remarked earlier,  the proof of
Theorem~\ref{birational-domino}, also shows that the $V$-torsion
of $H^2(W(\O_X))$ and $H^2(W(\O_{X'}))$ are isomorphic. Then by
\cite[Proposition 6.6, Page 642]{illusie79b} we see that the
$V$-torsion of $X$ and $X'$ are isomorphic. Thus we have an
isomorphism on the $V$-torsion $H^2_{cris}(X/W)_{v}\isom
H^2_{cris}(X'/W)_v$. Further it is standard that the
N\'eron-Severi group of $X$ does not acquire any torsion under
blowup $X'\to X$. So we have an isomorphism
$$H^2_{cris}(X/W)_{d}\to H^2_{cris}(X'/W)_{d},$$
of the divisorial torsion of $X$ and $X'$. Therefore in the
commutative diagram
$$
\xymatrix{
  0 \ar[r]^{} & H^2_{cris}(X/W)_d \ar[d]^{\simeq} \ar[r]  & H^2_{cris}(X/W)_{Tor}
  \ar[d]_{\simeq} \ar[r]^{} & H^2_{cris}(X/W)_e \ar[d]_{}
  \ar[r]^{} & 0 \\
  0  \ar[r]^{} & H^2_{cris}(X'/W)_d \ar[r] & H^2_{cris}(X'/W)_{Tor}
   \ar[r]^{} & H^2_{cris}(X'/W)_e \ar[r]^{} & 0
   }
$$
the first two columns are isomorphisms and the rows are exact so
that the last arrow is an isomorphism.
\end{proof}

\begin{remark} It is clear from
Proposition~\ref{torsion-species-and-blowups} that while studying
torsion in the crystalline cohomology of a surface that we can
replace $X$ by its smooth minimal model (when it exits).
\end{remark}
\subsection{$\kappa\leq0$ means no exotic torsion} The next result we want to prove is probably
well-known to the experts. But we  will prove a more precise form
of this result in Theorem~\ref{precise-torsion} and
Proposition~\ref{classification-of-torsion-for-kappa0}. We begin
by stating the result in its coarsest form.
\begin{theorem} Let $X/k$ be a smooth projective surface over a
perfect field. If $\kappa(X)\leq 0$ then $H^2_{\cris}(X/W)$ does
not have exotic torsion.
\end{theorem}
\subsection{The case $\kappa(X)=-\infty$}
The case $\kappa(X)=-\infty$ is the easiest of all. If
$\kappa(X)=-\infty$, then $X$ is rational or ruled. If $X$ is
rational, by the birational invariance of torsion we reduce to the
case $X=\P^2$ or $X$ is a $\P^1$-bundle over $\P^1$ and in these
case on deduces the result following result by inspection. Thus
one has to deal with the case that $X$ is ruled.
\begin{proposition} Let $X$ be a smooth ruled surface over $k$. Then
$H^2_{\cris}(X/W)$ is torsion free and $X$ is Hodge-Witt.
\end{proposition}
\begin{proof}
The first assertion follows from the formula for crystalline
cohomology of a projective bundle over a smooth projective scheme. The
second assertion follows from the following lemma which is of
independent interest and will be of frequent use to us.
\end{proof}

\begin{lemma} \label{pgzero-hodge-witt}
Let $X$ be a smooth, projective variety over a perfect field $k$.
\begin{enumerate}
\item If $H^i(X,\O_X)=0$ then $H^i(X,W(\O_X))=0$.
\item If $X/k$ is a surface with $p_g(X)=0$ then
$X$ is Hodge-Witt.
\end{enumerate}
\end{lemma}
\begin{proof} This is well-known and was also was noted in
\cite{joshi00a}. We include it here for completeness. Clearly, it
is sufficient to prove the first assertion. We have the exact
sequence
\begin{equation}
0 \to W_{n-1}(\O_X)\to W_n(\O_X)\to \O_X \to 0
\end{equation}
The result now follows by induction on $n$ and the fact that
$H^i(X,\O_X)=0$.
\end{proof}

\begin{lemma}\label{pgzero-exotic}
Let $X$ be a smooth projective variety over a perfect field. If
$H^2(X,\O_X)=0$ then there is no exotic or $V$-torsion in
$H^2_{\cris}(X/W)$.
\end{lemma}
\begin{proof}
By Illusie's description of exotic torsion (see \cite{illusie79b})
one knows that it is the quotient of a part of $p$-torsion in
$H^2(X,W(\O_X))$, but this group is zero by the
Lemma~\ref{pgzero-hodge-witt}, so its quotient by the $V$-torsion
is zero as well.
\end{proof}

\subsection{Surfaces with $\kappa(X)=0$}
        Let $X$ be a smooth projective surface with $\kappa(X)=0$.
We can describe the  crystalline torsion of such surfaces
completely. The description  of surfaces with $\kappa(X)=0$ breaks
down in to the following cases based on the value of $b_2$ of the
surface $X$ (see \cite{bombieri77}).
\begin{proposition}\label{classification-of-torsion-for-kappa0}
Let $X/k$ be a smooth projective surface of Kodaira dimension
zero. Then one has the following:
\begin{enumerate}
\item if $b_1(X)=4$, then $X$ is an abelian surface and
$H^2_{\cris}(X/W)$ is torsion free so all species of torsion are
zero; moreover $X$ is Hodge-Witt if and only if $X$ has $p$-rank one.
\item if $b_2(X)=22$, then $X$ is a $K3$-surface and
$H^2_{\cris}(X/W)$ is torsion free and $X$ is Hodge-Witt if and only
if the formal Brauer group of $X$ is of finite height.
\item Assume $b_2=2$. Then $b_1=2$ and there are two subcases given by
the value of $p_g$:
\begin{enumerate}
\item if $p_g=0$, then $H^2_{\cris}(X/W)$ has no torsion and $X$ is
Hodge-Witt;
\item if $p_g=1$ then $H^2_{\cris}(X/W)$ has $V$-torsion and $\Pic(X)$
is not reduced.
\end{enumerate}
\item if $b_2=10$, then $p_g=0$ and unless $\text{char}(k)=2$ and in
the latter case $p_g=1$; in the former case $X$ is Hodge-Witt and $X$
has no $V$-torsion; if $p_g=1$ then $H^2_{\cris}(X/W)$ has $V$-torsion.
\end{enumerate}
\end{proposition}
\begin{proof}
The assertion (1) is well-known. The assertion (2) is due to
\cite{nygaard79b}. The cases when $X$ has $p_g=0$ can be easily
dealt with by using Lemma~\ref{pgzero-hodge-witt} and
Lemma~\ref{pgzero-exotic}.
\end{proof}

\begin{corollary} Let $X$ be a smooth projective surface over a
perfect field. Assume $\kappa(X)=0$ then $H^2_{\cris}(X/W)$ has no
exotic torsion.
\end{corollary}
\begin{proof}
The cases when $p_g=0$ are treated by means of
Lemma~\ref{pgzero-exotic}. The remaining cases follow from Suwa's
criterion (see \cite{suwa83}) as in all these case one has by
\cite{bombieri77} that $q=-p_a$ so Suwa's criterion applies and in
this situation $H^2(X,W(\O_X))$ is $V$-torsion, and therefore there is
no exotic torsion
\end{proof}

\begin{corollary}\label{precise-torsion}
Let $X$ be a smooth, projective surface over an algebraically
closed field $k$ of characteristic $p>0$.
\begin{enumerate}
\item If $X$ has exotic torsion the $\kappa(X)\geq 1$.
\item If $X$ has $V$-torsion then
\begin{enumerate}
\item  $\kappa(X)\geq 1$ or,
\item  $\kappa(X)=0$ and $X$ has $b_2=2, p_g=1$ or $p=2$, $b_2=10,
p_g=1$.
\end{enumerate}
\end{enumerate}
\end{corollary}

\subsection{A criterion for non-existence of exotic torsion}
Apart from \cite{joshi00b} and \cite{suwa83} we do not know any
useful general criteria for ruling out existence of exotic
torsion. The following trivial result is often useful in dealing
with exotic torsion in surfaces of general type.
\begin{proposition} Let $X/k$
be a smooth, projective surface over a perfect field. Assume
$\Pic(X)$ is reduced and $H^2(X,W(\O_X))$ is of finite type. Then
$H^2_{\cris}(X/W(k))$ does not contain exotic torsion.
\end{proposition}
\begin{proof}
        Recall from \cite{nygaard79b} that a smooth projective
surface is Hodge-Witt if and only if $H^2(X,W(\O_X))$ is of finite
type. Then as $\Pic(X)$ is reduced, we see that $V$ is injective
on $H^2(X,W(\O_X))$. Thus $H^2(X,W(\O_X))$ is a Cartier module of
finite type. By \cite[Proposition 2.5, page 99]{illusie83b} we
know that any $R^0$-module which is a finite type $W(k)$-module is
a Cartier module if and only if it is a free $W(k)$-module. Thus
$H^2(X,W(\O_X))$ is a free $W(k)$-module of finite type.  By
\cite[Section 6.7, page 643]{illusie79b} we see that the exotic
torsion of $H^2_{\cris}(X/W(k))$ is zero as it is a quotient of
the image of torsion in $H^2_{\cris}(X/W(k))$ (under the canonical
projection $H^2_{\cris}(X/W(k))\to H^2(X,W(\O_X))$) by the
$V$-torsion of $H^2(X,W(\O_X))$. But as $H^2(X,W(\O_X))$ is
torsion free, we see that the exotic torsion is zero.
\end{proof}

\section{Mehta's question for surfaces}\label{mehta-question}
\subsection{Is torsion uniformizable?} In this section we answer the following question of
Mehta (see \cite{joshi00b}):
\begin{question} Let $X/k$ be a
smooth, projective, Frobenius split variety over a perfect field
$k$. Then does there exists a Galois \'etale cover $X'\to X$ such that
$H^2_{\cris}(X/W)$ is torsion free.
\end{question}
\subsection{Absence of exotic torsion}
In \cite{joshi00b} it was shown that the second crystalline cohomology
of smooth, projective, Frobenius split surface does not have exotic torsion in the
second crystalline cohomology. In \cite{joshi00a} it was shown that any
smooth, projective Frobenius split surface is ordinary.

\subsection{The case $\kappa(X)=0$}
We will prove now that the answer to the above question is
affirmative and in fact the assertion is true more generally for
$X$ with $\kappa(X)\leq 0$. The main theorems of this section are
\begin{theorem}\label{questofmehta1} Let $X$ be a smooth, projective surface of Kodaira
dimension at most zero, then there exists a Galois \'etale cover
$X'\to X$ such that $H^2_{\cris}(X/W)$ is torsion free.
\end{theorem}
\begin{theorem}\label{questofmehta2} Let $X$ be a smooth, projective
surface over a perfect field. Assume $X$ is Frobenius split. Then
there exists a Galois \'etale cover $X'\to X$ such that
$H^2_{\cris}(X'/W)$ is torsion free.
\end{theorem}
\begin{proof}{[of Theorem~\ref{questofmehta1}]}
We now note that Mehta's question is trivially true for ruled
surfaces as these have torsion-free crystalline cohomology. So we
may assume that $\kappa(X)=0$. In this case we have a finite
number of classes of surfaces for which the assertion has to be
verified. These classes are classified by $b_2$. When $X$ is a
$K3$ or an Enriques surface or an abelian surface then we can
take $X'=X$ as such surfaces have torsion free crystalline
cohomology. When $b_2=2$ the surface is bielliptic and by explicit
classification of these we know that we may take the Galois cover
to be the product of elliptic curves and so we are done in these
cases as well.
\end{proof}

\begin{proof}{[of Theorem~\ref{questofmehta2}]}
After Theorem~\ref{questofmehta1} it suffices to prove that the
Kodaira dimension of a Frobenius split surface is at most zero.
This follows from Proposition~\ref{classification-fsplit} below
(and is, in any case, well-known to experts).
\end{proof}

\begin{proposition}\label{classification-fsplit} Let $X$ be a smooth
projective surface. If $X$ is a Frobenius split then, $X$ has
Kodaira dimension at most zero and is in the following list:
\begin{enumerate}
\item $X$ is either rational or ruled over an ordinary curve,
\item $X$ is a either an ordinary $K3$, or an ordinary abelian surface
or $X$ is bielliptic with an ordinary elliptic curve as its Albanese
variety, or $X$ is an ordinary Enriques surface.
\end{enumerate}
\end{proposition}

\begin{proof}
We first control the Kodaira dimension of a Frobenius split
surface. By \cite{mehta85b} we know that if $X$ is Frobenius
split, then $H^2(X,K_X)\to H^2(X,K_X^p)$ is injective, or by
duality, $H^0(X,K_X^{1-p})$ has a non-zero section and hence in
particular, $H^0(X,K_X^{-n})$ has sections for large $n$. Hence,
if $\kappa(X)\geq 1$, then as the pluricanonical system $P_n$ is
also non-zero for large $n$, so we can choose an $n$ large enough
such both that $K_X^n$ and $K_X^{-n}$ have sections and so
$K_X^n=\O_X$ for some integer $n$. But this contradicts the fact
that $\kappa(X)=1$, for in that case $K_X$ is non-torsion, so we
deduce that $X$ has $\kappa(X)\leq 0$. Now the result follows from
the classification of surfaces with $\kappa(x)\leq 0$.
\end{proof}

\section{Hodge-Witt numbers of threefolds}\label{hodge-witt-numbers-of-threefolds}
In this section we compute Hodge-Witt numbers of smooth proper
threefolds. In Theorem~\ref{non-liftable-b3} we characterize
Calabi-Yau threefolds with negative Hodge-Witt numbers and in
Proposition~\ref{hirokado-example} Calabi-Yau threefolds constructed by \cite{hirokado99} and \cite{schroer04} appear as
examples of Calabi-Yau threefolds with negative Hodge-Witt numbers.
\subsection{Non-negative Hodge-Witt numbers of threefolds} We begin by listing all the Hodge-Witt numbers of a
smooth, proper threefolds which are always non-negative.
\begin{proposition}\label{non-negative-hw}
        Let $X/k$ be a smooth, proper threefold over a perfect
field of characteristic $p>0$.
\begin{enumerate}
\item Then $h^{i,j}_W\geq 0$ except possibly when $(i,j)\in\left\{
(1,1),(2,1),(1,2),(2,2)\right\}$.
\item All the Hodge-Witt numbers except
$h_W^{1,1}=h^{2,2}_W,h_W^{1,2}=h_W^{2,1}$ coincide with the
corresponding slope numbers.
\item For the exceptional cases we
have the following formulas.
\begin{eqnarray}
h^{1,2}_W&=&m^{1,2}-T^{0,3}\\
h^{1,1}_W&=&m^{1,1}-2T^{0,2}
\end{eqnarray}
\end{enumerate}
\end{proposition}

\begin{proof}
Let us prove (1). This  uses the criterion for degeneration of the
slope spectral sequence given in \cite{joshi00b}. The criterion
shows that $T^{i,j}=0$ unless $(i,j)\in
\left\{(0,3),(0,2),(1,2),(3,1)\right\}$. By Definition~\ref{hodge-witt-definition} of $h^{i,j}_W$ it suffices to verify that
$T^{i-1,j+1}=0$ except possibly in the four cases listed in the
proposition. This completes the proof of (1). To prove (2), we
begin by observing that Hodge-Witt symmetry
\ref{hodge-witt-symmetry} gives $h^{2,1}_W=h^{1,2}_W$ and we also
have $h^{1,1}_W=h^{3-1,3-1}_W=h^{2,2}_W$. So this proves the first
part of (2). Next the criterion for degeneration of the slope
spectral sequence shows that in all the cases except the listed
ones, the domino numbers which appear in the definition of
$h^{i,j}_W$ are zero. This proves (2). The second formula of (3)
now follows again from the definition of $h^{i,j}_W$  (see
\ref{hodge-witt-definition} and the criterion for the degeneration
of the slope spectral sequence). The first formula of (3) follows
from the definition of $h^{1,2}_W=m^{1,2}+T^{1,2}-2T^{0,3}$, and
by duality for domino numbers \ref{domino-duality} we have
$T^{1,2}=T^{0,3}$.
\end{proof}

\subsection{Hodge-Witt Formulaire for Calabi-Yau threefolds}
The formulas for Hodge-Witt numbers can be  made even more
explicit in the case of Calabi-Yau varieties.
\begin{proposition}\label{calabi-yau-formulaire}
Let $X$ be a smooth, proper Calabi-Yau threefold. Then the
Hodge-Witt numbers of $X$ are given by:
\begin{eqnarray}
h^{0,0}_W&=&1\\
h^{0,1}_W&=&0\\
h^{0,2}_W&=&0\\
h^{0,3}_W&=&1\\
h^{1,1}_W&=&b_2\\
h^{1,2}_W&=&b_2-\frac{1}{2}c_3(X)\\
h^{1,3}_W&=&0
\end{eqnarray}
The remaining numbers are computed from these by using Hodge-Witt
symmetry and the symmetry $h^{i,j}_W=h^{3-i,3-j}_W$.
\end{proposition}
\begin{proof}
We first note that $h^{0,0}_W=h^{0,0}=1$ is trivial. The
Hodge-Witt numbers in the first four equations are also
non-negative by the previous proposition as $T^{0,2}=0$. Moreover,
by \cite{ekedahl-diagonal} it suffices to note that $h^{i,j}_W\leq
h^{i,j}$ and in the second and the third formulas we have by
non-negativity of $h^{i,j}_W$ that $0\leq h^{0,1}_W\leq h^{0,1}=0$
(by the definition of Calabi-Yau threefolds) and similarly for the
third formula. The fourth formula is a consequence of Crew's
formula and first three equations:
\begin{equation}
0=\chi(O_X)=h^{0,0}_W-h^{0,1}_W+h^{0,2}_W-h^{0,3}_W
\end{equation}
In particular we deduce  from the fourth formula and
$$0\leq h^{0,3}_W=m^{0,3}+T^{0,3}\leq 1$$
that if  $T^{0,3}=0$ so that $X$ is Hodge-Witt then the definition
of $m^{0,3}$ shows that
$$m^{0,3}=\sum_{\lambda}(1-\lambda)\dim
H^3_{cris}(X/W)_{[\lambda]}.$$ So that the inequality shows that
$H^3_{cris}(X/W)$ contains at most one slope $0\leq \lambda<1$
with $\lambda=\frac{h-1}{h}$ (with $h$ allowed to be $1$, to
include $\lambda=0$), and so if $T^{0,3}=0$ then $m^{0,3}=1$. Thus
it remains to prove the formulas for $h^{1,1}_W$ and $h^{2,1}_W$.
We first note that by definition:
\begin{equation}
h^{1,1}_W=m^{1,1}+T^{1,1}-2T^{0,2}.
\end{equation}
Now as $h^{0,2}=0$ we get $T^{0,2}=0$, and $T^{1,1}=0$ by
\cite[Corollaire 3.11, page 136]{illusie83b}. Thus we get
$h^{1,1}_W=m^{1,1}$. Next
$$m^{0,2}+m^{1,1}+m^{2,0}=b_2$$
and as $m^{0,2}=0=m^{2,0}$ we see that $h^{1,1}_W=m^{1,1}=b_2$.
The remaining formula is also a straight forward application of
Crew's formula
\begin{equation}
\chi(\Omega^1_X)=h^{1,0}_W-h^{1,1}_W+h^{1,2}_W-h^{1,3}_W
\end{equation}
and the Grothendieck-Hirzebruch-Riemann-Roch for $\Omega^1_X$,
which we recall in the following lemma.
\end{proof}

\begin{lemma}
Let $X$ be a smooth proper threefold over a perfect
field. Then
\begin{equation}
\chi(\Omega^1_X)= 
-\frac{23}{24}c_1\cdot
c_2-\frac{1}{2}c_3.
\end{equation}
\end{lemma}
\begin{proof}
This is trivial from the Grothendieck-Hirzebruch-Riemann-Roch
theorem. We give a proof here for completeness. We have
\begin{eqnarray*}
\chi(\Omega^1_X)&=&\left[3-c_1+\frac{1}{2}(c_1^2-2c_2)+
\frac{1}{6}\left(-c_1^3-3c_1\cdot c_2-3c_3\right)\right]\\
&&\quad\times
\left[1+\frac{1}{2}c_1+\frac{1}{12}(c_1^2+c_2)+\frac{1}{24}c_1\cdot
c_2\right]_3
\end{eqnarray*}
This simplifies to the claimed equation.
\end{proof}

\subsection{Calabi-Yau threefolds with negative $h^{1,2}_W$}
\label{calabi-yau-negative} 
In
this section we investigate Calabi-Yau threefolds with negative
Hodge-Witt numbers. From the formulas
\eqref{calabi-yau-formulaire} it is clear that the only possible
Hodge-Witt number which might be negative is $h^{1,2}_W$. We begin
by characterizing such surfaces (see Theorem~\ref{non-liftable-b3}
below). Then we verify (in Proposition~\ref{hirokado-example})
that the in characteristic $p=2,3$, there do exist Calabi-Yau
threefolds with negative Hodge-Witt numbers. These are the
Hirokado and Schr\"oer Calabi-Yau threefolds (which do not lift to
characteristic zero).

\begin{theorem}
\label{non-liftable-b3}
Let $X$ be a smooth, proper Calabi-Yau threefold over a
perfect field of characteristic $p>0$. Then the following
conditions are equivalent
\begin{enumerate}
\item the Hodge-Witt number $h^{1,2}_W=-1$,
\item the Hodge-Witt number  $h^{1,2}_W<0$,
\item the $W$-module $H^3_{cris}(X/W)$ is torsion,
\item the Betti number $b_3=0$.
\item the threefold $X$ is not Hodge-Witt and the slope number
$m^{1,2}=0$.
\end{enumerate}
\end{theorem}
\begin{proof}
It is clear that (1) implies (2), and similarly it is clear that
(3) $\Leftrightarrow$ (4). So the only assertions which need to be
proved are the assertions (2) implies (3), (4) implies (5) and the assertion (4)
implies (1). So let us prove (2) implies (3). By the proof of
\ref{calabi-yau-formulaire}  we see that
$h^{1,2}_W=m^{1,2}-T^{0,3}$ and as $T^{0,3}\leq 1$, we see that if
$h^{1,2}_W<0$ then we must have $h^{1,2}_W=-1, T^{0,2}=1,
m^{1,2}=0$ (the first of these equalities of course shows that (2)
implies (1)). So the hypothesis of (2) implies in particular that
$T^{0,3}=1$ in other words, $X$ is non-Hodge-Witt and so
$H^3(X,W(\O_X))$ is $p$-torsion. Hence  the number $m^{0,3}=0$.
Now by the symmetry \eqref{slope-symmetry} we see that
$m^{0,3}=m^{3,0}=m^{1,2}=m^{1,2}=0$. From this and the formula
\eqref{slope-betti} we see that
$$b_3=m^{0,3}+m^{1,2}+m^{2,1}+m^{3,0}=0.$$
This completes the proof of (2) implies (3). Let us prove that (4)
implies (5). The hypothesis of (4) and preceding equation shows
that $m^{0,3}=m^{1,2}=0$. So we have to verify that $X$ is not
Hodge-Witt. Assume that this is not the case. The vanishing
$m^{0,3}=0$ says that $H^3(X,W(\O_X))\tensor_W K=0$ and as
$H^2(X,\O_X)=0$ we see that $V$ is injective on $H^3(X,W(\O_X))$.
If  $X$ is Hodge-Witt, then  so this  $W$-module is a finite type
$W$-module with $V$-injective. Therefore it is a Cartier module of
finite type. By \cite{illusie83b} such a Cartier module is a free
$W$-module. Hence $H^3(X,W(\O_X))$ is free  and torsion so we
deduce that $H^3(X,W(\O_X))$ is zero. But as
$H^3(X,W(\O_X))/VH^3(X,W(\O_X))=H^3(X,\O_X)\neq 0$. This is a
contradiction. So we see that (4) implies (5). So now let us prove
that (5) implies (1).  The first hypothesis of (5) implies that
$X$ is a non Hodge-Witt Calabi-Yau threefold so that $T^{0,3}=1$
and hence we see that $h^{1,2}_W=m^{1,2}-T^{0,3}=-1<0$. This
completes the proof of the theorem.
\end{proof}

Let us record the following trivial but important corollary of this Theorem~\ref{non-liftable-b3}.
\begin{corollary}\label{cor:calabi-yau-h12-lower-bound}
Let $X$ be any smooth, proper, Calabi-Yau threefold. Then one has
$$ h^{1,2}_W\geq -1$$
equivalently
$$b_2\geq\frac{1}{2}c_3-1.$$
\end{corollary}

The following assertion shows that there do exist Calabi-Yau threefolds with $h^{1,2}_W<0$:
\begin{proposition}\label{hirokado-example}
Let $k$ be an algebraically closed field of characteristic
$p=2,3$. Then there exists  smooth, proper Calabi-Yau
threefold $X$ such that $h^{1,2}_W<0$.
\end{proposition}
\begin{proof}
 In \cite{hirokado99} (resp. \cite{schroer04}) M.~Hirokado and (resp. Stefan Schr\"oer) have constructed an
examples of Calabi-Yau (and in fact, families of such threefolds
in the latter case) threefold in characteristic $p=2,3$ which are
not liftable to characteristic zero.   We claim that these
Calabi-Yau threefolds are the examples we seek. It was verified in
loc. cit. that these threefolds have $b_3=0$. So we are done by
Theorem~\ref{non-liftable-b3}.
\end{proof}

\begin{corollary}
The Hirokado and Schr\"oer threefolds are not Hodge-Witt.
\end{corollary}

\begin{remark}
The Hirokado and Schr\"oer threefolds have been investigated in
detail by Ekedahl (see \cite{ekedahl04a}) who has proved their
arithmetical rigidity. One should note that the right hand side of
\eqref{calabi-yau-formulaire} is non-negative if $X$ lifts to
characteristic zero without any additional assumptions on torsion
of $H^*_{cris}(X/W)$ as the following proposition shows.
\end{remark}

\begin{proposition}
Let $X$ be a smooth, proper
Calabi-Yau threefold. If $X$ lifts to characteristic zero then
%
$$c_3\leq 2b_2.$$
\end{proposition}
\begin{proof}
Under the hypothesis, we know that  $b_1=b_5=0$ so that
$c_3=2+2b_2-b_3$, so that $b_2-\frac{1}{2}c_3=\frac{b_3}{2}-1$ and
by the Hodge decomposition, $b_3\geq 2$ is even and so the
assertion holds.
\end{proof}

\begin{remark}
The examples (for $p=2,3$) constructed in \cite{schroer04} has $b_2=23,c_3=48$ so $c_3>2b_2$. For the example of \cite{hirokado99} we have also have $c_3>2b_2$.
\end{remark}
\subsection{Classical and non-classical Calabi-Yau threefolds}
The above results motivate the following definition. Let $X$ be a Calabi-Yau threefold over a perfect field of characteristic $p>0$. Then we say that $X$ is a \emph{classical Calabi-Yau threefold} if $h^{1,2}_W\geq 0$, other wise we say that $X$ is \emph{non-classical Calabi-Yau threefold}.

Note that by \eqref{calabi-yau-formulaire} any non-classical Calabi-Yau threefold is not Hodge-Witt i.e $T^{0,3}\neq0$.

It is clear from the definition and the above results that the examples of \cite{hirokado99,schroer04} are non-classical Calabi-Yau threefolds. In particular non-classical Calabi-Yau threefolds exist (in general).

\subsection{Hodge-Witt rigidity}
For a Calabi-Yau threefolds recall that by Serre duality, and using the perfect pairing $\Omega^1\tensor \Omega^2_X \to \Omega^3_X=\O_X$ one has for any Calabi-Yau
\be 
h^{1,2}=0\iff \dim H^1(X,T_X)=0.
\ee
If $X$ satisfies $h^{1,2}=0$ (equivalently $H^1(X,T_X)=0$) then one says that $X$ is a \emph{rigid} Calabi-Yau threefold. Using the Hodge decomposition theorem, it is easy to verify that rigidity of a Calabi-Yau threefold over complex numbers is equivalent to the condition $b_2=\frac{1}{2}c_3$.

Motivated by the above definition of rigidity we introduce a weaker notion which is valid over fields of characteristic $p>0$. Let $X/k$ be a smooth, Calabi-Yau threefold. We say that $X$ is \emph{Hodge-Witt rigid} if $$h_W^{1,2}=0.$$ From Proposition~\ref{calabi-yau-formulaire} we see that 
\be 
 X \text{ is Hodge-Witt rigid} \iff b_2=\frac{1}{2}c_3.
\ee 
If $X$ is a classical Calabi-Yau threefold which is rigid then $0=h^{1,2}\geq h^{1,2}_W\geq 0$ shows that $h^{1,2}_W=0$. In other words
\be 
X \text{ is classical and rigid}\implies X \text{ is Hodge-Witt rigid}.
\ee 
But in general the two notions are different: for example the Hirokado threefold is rigid but not Hodge-Witt rigid. However the following result shows that if $X$ is a Mazur-Ogus threefold then Hodge-Witt rigidity and rigidity are equivalent.
\begin{proposition}\label{pro:hodge-witt-rigidity}
Let $X$ be a Mazur-Ogus, Calabi-Yau threefold over a perfect field of characteristic $p>0$. Then the following are equivalent:
\benum 
\item $X$ is Hodge-Witt rigid i.e. $h_W^{1,2}=0$.
\item $X$ is rigid i.e. $h^{1,2}=0$.
\item $b_3=2$.
\item $c_3=2b_2$
\eenum
\end{proposition}
\bp 
The assertion is clear from Proposition~\ref{calabi-yau-formulaire} and the fact that if $X$ is Mazur-Ogus then $h^{1,2}=h_W^{1,2}$.
\ep

\begin{remark}
Let us note that the there exists Calabi-Yau threefolds which are Hodge-Witt rigid. Indeed by the above proposition it is enough to find Calabi-Yau threefolds which are Mazur-Ogus and rigid. This is not too difficult as the reduction modulo any sufficiently large, unramified prime of good reduction of any rigid  Calabi-Yau threefold is both rigid and Mazur-Ogus (as it arises from characteristic zero).
\end{remark}

\subsection{Geography of Calabi-Yau Threefolds}\label{Calabi-Yau-Geography}
\pgfplotsset{every tick label/.append style={font=\tiny}}
\begin{figure}
\centering
\subfloat{
\begin{tikzpicture}[scale=1]
\begin{axis}[
axis lines=middle,
xlabel={\tiny{$c_3$}},
ylabel={\tiny{$b_2$}},
legend style={legend pos=outer north east},
xmin=-2,
xmax=10,
ymin=-5,
ymax=5,
samples=100,
xtick distance=2,
ytick distance=2,
title={Figure 1: Geography of Calabi-Yau Threefolds over  $k=\C$},
] 
\node at  (axis cs:3,4) {\tiny{classical region $b_2\geq \frac{c_3}{2}$}};
\node at (axis cs:4,-1.5) {{\tiny{no Calabi-Yau's in $b_2< \frac{c_3}{2}-1$}}};
\node[rotate=25.8] at (axis cs:4,2.5) {\tiny{rigidity $h^{1,2}=h^{2,1}=0$}} ;
\addplot+[name path=classical,no marks,blue,domain=-2:10] {max((1/2)*x,1)};
\addlegendentry{\tiny{$b_2=\max(\frac{1}{2}c_3,1)$}}
\path[name path=lower](axis cs:-5,-4) -- (axis cs:10,-4);
\addplot[olive!10] fill between[ 
of = classical and lower];
\end{axis}
\end{tikzpicture}
}\qquad
\subfloat{
	\begin{tikzpicture}[scale=1]
	\begin{axis}[
	axis lines=middle,
	xlabel={\tiny{$c_3$}},
	ylabel={\tiny{$b_2$}},
	legend style={legend pos=outer north east},
	xmin=-2,
	xmax=10,
	ymin=-5,
	ymax=5,
	samples=100,
	xtick distance=2,
	ytick distance=2,
	title={Figure 2: Geography of Calabi-Yau Threefolds},
	] 
	\node at  (axis cs:3,4) {\tiny{classical region $b_2\geq \frac{c_3}{2}$}};
	\node at (axis cs:4,-1.5) {{\tiny{no Calabi-Yau's in $b_2< \frac{c_3}{2}-1$}}};
	\node[rotate=25.8] at (axis cs:4,2.5) {\tiny{Hodge-Witt rigidity $h^{1,2}_W=0$}} ;
	\node[rotate=25.8] at (axis cs:8,2.5) {\tiny{$h^{1,1}_W=-1$}} ;
	\node[rotate=25.8] at (axis cs:6,2.5) {\tiny{no Calabi-Yau in this region}} ;
	\addplot+[name path=classical,no marks,blue,domain=-2:10] {max((1/2)*x,1)};
	\addlegendentry{\tiny{$b_2=\max(\frac{1}{2}c_3,1)$}}
	\addplot+[name path=nonclassical,no marks,red,domain=4:10] {max((1/2)*x-1,1)};
	\addlegendentry{\tiny{$b_2=\max(\frac{1}{2}c_3-1,1)$}}
	\path[name path=lower](axis cs:-5,-4) -- (axis cs:10,-4);
	\addplot[olive!10] fill between[ 
	of = classical and lower];
	\end{axis}
	\end{tikzpicture}
}
\end{figure}
Note that for any Calabi-Yau threefold, $b_2,c_3$ determine $b_3$. Thus one should view $b_2,c_3$ as variables for studying geography of Calabi-Yau threefolds.
Preceding results establish some results in the subject of geography of Calabi-Yau threefolds. 
These results are  best summarized in the two maps  shown in Figure~1 (for $k=\C$) and in Figure~2 (for $k$ perfect of characteristic $p>0$). In either of these maps a quintic Calabi-Yau threefold in $\P^4$  such as $x^5_0+\cdots+x_4^5=0$, which is a classical, Mazur-Ogus Calabi-Yau threefold, corresponds to the point $(-200,1)$ on the line $b_2=\max(\frac{1}{2}c_3,1)$. Rigid Calabi-Yau threefolds live on the line $b_2=\frac{1}{2}c_3$. \emph{Over complex numbers geography of Calabi-Yau threefolds does not have a non-classical component as $h^{1,2}\geq 0$.} 

Now suppose that we are in postive characteristic. The geography now looks different from that over $k=\C$. One can have non-classical Calabi-Yau threefolds and if these exist then such threefolds live on the line $b_2=\frac{1}{2}c_3-1$. In particular the Hirokado and Schr\"oer threefolds are on the line $b_2=\frac{1}{2}c_3-1$.  The region $\frac{1}{2}c_3-1<b_2<\frac{1}{2}c_3$ is unpopulated as in this region one would have $-1<h^{1,2}_W=b_2-\frac{1}{2}c_3<0$ which is impossible as $h^{1,2}_W$ is an integer.
Note that if $X$ is a non-classical Calabi-Yau threefold then $h^{1,2}_W=-1=b_2-\frac{1}{2}c_3$ from which we see that $c_3>0$ and as $b_2\geq 1$ by projectivity, so one has $c_3\geq 4$. \emph{In particular one deduces that the half line $b_2=\frac{1}{2}c_3-1, c_3\geq 4$, which corresponds to non-classical Calabi-Yau threefolds, is an island in the geography of Calabi-Yau threefolds} as it is unconnected from the continent $b_2\geq\frac{1}{2}c_3$ which is the realm of classical Calabi-Yau threefolds. The line $b_2=\frac{1}{2}c_3$ corresponds to Hodge-Witt rigidity $h^{1,1}_W=0$ and is populated, for instance, by Mazur-Ogus rigid Calabi-Yau threefolds.

\subsection{A Conjecture about Calabi-Yau threefolds}
Let $X$ be a smooth proper Calabi-Yau threefold over a perfect field of characteristic $p>0$. The following conjecture provides a necessary and sufficient condition for $X$ to admit a lifting to characteristic zero.  Note that in characteristic zero any Calabi-Yau threefold does not have non-vanishing global one forms (this is a consequence of Hodge symmetry) on the other hand in positive characteristic we do not know if this vanishing assertion always holds. On the other hand if $H^*_{cris}(X/W)$ is torsion-free then certainly $H^0(X,\Omega^1_X)=0$ on the other hand if $b_3=0$ then $H^3_{cris}(X/W)$ is torsion (possibly zero). As had been pointed out in the preceding discussion, any Calabi-Yau threefold which lifts to characteristic zero is classical. So the following is a very optimistic conjecture:
\begin{conj}\label{calabi-yau-conjecture}
Let $X$ be a smooth, proper Calabi-Yau threefold. Then $X$ lifts to characteristic zero if and only  if
\benum
	\item $H^0(X,\Omega^1_X)=0$ and
	\item $h^{1,2}_W\geq 0$, equivalently:
	\item $c_3 \leq 2b_2$, equivalently:
	\item $X$ is classical.
\eenum
\end{conj}
Note that equivalence of (2) and (3) i.e. $h^{1,2}_W\geq 0 \iff c_3\leq 2b_2$ is clear from Proposition~\ref{calabi-yau-formulaire}. 
Also note that if $X$ is Hodge-Witt then $h^{1,2}_W\geq0$. If  $X$ is Hodge-Witt and also has torsion-free crystalline cohomology then $X$ is Mazur-Ogus by \cite[Theorem 4.7, page 204]{illusie83b}. Thus the Conjecture
~\ref{calabi-yau-conjecture} predicts that any Hodge-Witt Calabi-Yau threefold with torsion free $H^*_{cris}(X/W)$ lifts to characteristic zero. 

\subsection{A remarkable theorem of F.~Yobuko}\label{sec:yobuko-theorem}
At the time we made this conjecture there was not much evidence for it. But recently \cite{yobuko16} has proved the following remarkable Theorem (the formulation provided below is  our reformulation of \cite{yobuko16} for Calabi-Yau Threefolds)  which provides some evidence Conjecture~\ref{calabi-yau-conjecture}.

\bthm\label{th:yobuko-reformulation}
Suppose $X$ is a smooth, proper Hodge-Witt Calabi-Yau threefold. Consider the following assertions.
\benum[label={(\bf\arabic{*})}]
\item\label{th:yobuko-reformulation-1} $X$ is quasi-Frobenius split.
\item\label{th:yobuko-reformulation-2} $X$ has finite height.
\item\label{th:yobuko-reformulation-3} $X$ is Hodge-Witt (hence classical).
\item\label{th:yobuko-reformulation-4} $X$ lifts to $W_2$. 
\eenum
Then one has $\ref{th:yobuko-reformulation-1}\iff\ref{th:yobuko-reformulation-2}\iff\ref{th:yobuko-reformulation-3}\implies\ref{th:yobuko-reformulation-4}$.
In particular any Hodge-Witt Calabi-Yau threefold lifts to $W_2$.
\ethm
\bp
The implication $\ref{th:yobuko-reformulation-1}\iff\ref{th:yobuko-reformulation-2}$ is due to \cite{yobuko16} and  the implication $\ref{th:yobuko-reformulation-1}\implies\ref{th:yobuko-reformulation-4}$ is the main theorem of \cite{yobuko16}. The assertion $\ref{th:yobuko-reformulation-3}\implies\ref{th:yobuko-reformulation-2}$ is  standard: Hodge-Witt hypothesis implies $H^3(W(\O_X))$ is free of finite type  (as $H^2(X,\O_X)=0$) so $X$ is of finite height.  Now $\ref{th:yobuko-reformulation-2}\implies\ref{th:yobuko-reformulation-3}$ follows from \cite{ekedahl-diagonal}, \cite{joshi00b} as there is only one possibly non-trivial domino number $T^{0.3}$. If $X$ has finite height then $H^3(W(\O_X))$ is free of finite type and hence $T^{0,3}=0$. So $X$ is Hodge-Witt. .
\ep

\begin{remark}
Let us remark that for Calabi-Yau variety $X$ of dimension $\leq 3$, $X$ is Hodge-Witt if and only if $X$ is of finite height. If $\dim(X)>3$ and $X$ is Hodge-Witt then $X$ is of finite height but the converse may not hold. The main theorem of \cite{yobuko16} proves more generally that if $X$ is of finite height then $X$ lifts to $W_2$. In particular it follows that any Hodge-Witt Calabi-Yau variety (of any dimension) lifts to $W_2$.
\end{remark}

\begin{remark}
Let us remark that if $k=\C$, then $c_3=2b_2$ holds if and only if $X$ is a rigid Calabi-Yau threefold. Indeed if $k=\C$, $h^{1,2}=b_2-\frac{1}{2}c_3=0$ if and only if $H^2(X,\Omega^1_X)=0$. By Serre duality and the fact that $X$ is a Calabi-Yau threefold we get $$H^2(X,\Omega^1_X)=H^1(X,T_X)=0.$$
So $X$ is rigid and conversely.
\end{remark}

\begin{remark}
Let us point out that the hypothesis $h^{1,2}_W\geq0\Longrightarrow b_3\neq 0$. To see this we use our formula
$$
h^{1,2}_W=m^{1,2}-T^{0,3}.
$$
We see that the assertion is immediate if $m^{1,2}\geq 1$, as $b_3=m^{0,3}+m^{1,2}+m^{2,1}+m^{3,0}$. So assume $m^{1,2}=0$. As $h^{1,2}_W\geq 0$, we see that $m^{1,2}=0$ gives $T^{0,3}=0$. Thus $X$ is a Hodge-Witt Calabi-Yau threefold. Thus $H^3(W(\O_X))$ is of finite type and $V$ is injective on it and hence $H^3(W\O_X))$ is torsion free and its $W$-rank is at least the length of $H^3(W(\O_X))/VH^3(W(\O_X))=H^3(\O_X)\neq 0$, so $H^3(W\O_X))\tensor_W K\neq 0$, so $H^3_{cris}(X/W)\tensor K\neq 0$ which gives $b_3\neq 0$. 
\end{remark}

\bibliographystyle{plain}

\end{document}